\documentclass{amsart}

\usepackage{hyperref}
\hypersetup{
    colorlinks=true,
    linkcolor=blue,
    citecolor=blue,
    urlcolor=blue,
}

\makeatletter
\newcommand\org@maketitle{}
\newcommand\@authors{}
\let\org@maketitle\maketitle
\def\maketitle{%
	\let\@authors\authors
	\nxandlist{; }{ and }{; }\@authors
	\hypersetup{
		linktocpage=true,
		pdftitle={\@title},
                pdfauthor={\@authors},
                pdfsubject={\subjclassname. \@subjclass},
		pdfkeywords={\@keywords}
	}%
	\org@maketitle
}
\makeatother

\usepackage{amssymb, accents}
\usepackage{rsfso}
\usepackage{mathtools}
\usepackage{microtype}
\usepackage[alphabetic,msc-links]{amsrefs}
\usepackage{enumitem}
\usepackage{amsfonts}
\DeclareMathAlphabet{\mathcal}{OMS}{cmsy}{m}{n}

\usepackage{doi}
\renewcommand{\PrintDOI}[1]{\doi{#1}}

\numberwithin{equation}{section}

\newtheorem{theorem}{Theorem}[section]
\newtheorem{lemma}[theorem]{Lemma}

\theoremstyle{definition}
\newtheorem{definition}[theorem]{Definition}

\theoremstyle{remark}
\newtheorem{remark}[theorem]{Remark}

\newcommand{\cH}{{\mathcal H}}
\newcommand{\cD}{{\mathcal D}}
\newcommand{\cQ}{{\mathcal Q}}

\newcommand{\al}{\alpha}
\newcommand{\be}{\beta}
\newcommand{\g}{\gamma}
\newcommand{\de}{\delta}

\newcommand{\la}{\lambda}
\newcommand{\ka}{\kappa}

\newcommand{\R}{\mathbb{R}}
\newcommand{\vp}{\varphi}

\newcommand{\bb}{\mathbf{b}}
\newcommand{\bg}{\mathbf{g}}
\newcommand{\bh}{\mathbf{h}}
\newcommand{\bff}{\mathbf{f}}
\newcommand{\bq}{\mathbf{q}}
\newcommand{\bA}{\mathbb{A}}
\newcommand{\bbQ}{\mathbb{Q}}
\newcommand{\uu}{u_\alpha}
\newcommand{\vv}{v_\alpha}

\newcommand{\D}{\nabla}

\newcommand{\supp}{\operatorname{supp}}
\renewcommand{\div}{\operatorname{div}}

\newcommand{\mean}[1]{\langle #1\rangle}
\newcommand{\lmean}[1]{\left\langle #1\right\rangle}
\newcommand{\DMO}{\operatorname{DMO}}
\newcommand{\tr}{\operatorname{tr}}

\def\Xint#1{\mathchoice
  {\XXint\displaystyle\textstyle{#1}}%
  {\XXint\textstyle\scriptstyle{#1}}%
  {\XXint\scriptstyle\scriptscriptstyle{#1}}%
  {\XXint\scriptscriptstyle\scriptscriptstyle{#1}}%
  \!\int}
\def\XXint#1#2#3{{\setbox0=\hbox{$#1{#2#3}{\int}$}
    \vcenter{\hbox{$#2#3$}}\kern-.5\wd0}}

\def\dashint{\Xint-}

\mathchardef\ordinarycolon\mathcode`\:
\mathcode`\:=\string"8000
\begingroup \catcode`\:=\active
  \gdef:{\mathrel{\mathop\ordinarycolon}}
\endgroup

\author{Hongjie Dong}
\author{Seongmin Jeon}
\address[H. Dong]{Division of Applied Mathematics, Brown University, 182 George Street, Providence RI 02912, USA}
\email{hongjie\_dong@brown.edu}
\address[S. Jeon]{Department of Mathematics Education, Hanyang University, 222 Wangsimni-ro, Seongdong-gu, Seoul 04763, Republic of Korea}
\email{seongminjeon@hanyang.ac.kr}
\thanks{H. Dong was partially partially supported by the NSF under agreement DMS-2055244. S. Jeon was supported by the research fund of Hanyang University(HY-202400000003278).}

\title[Degenerate or singular parabolic systems]{Degenerate or singular parabolic systems with partially DMO coefficients: the Dirichlet problem}

\allowdisplaybreaks

\begin{document}
\begin{abstract}
In this paper, we study solutions $u$ of parabolic systems in divergence form with zero Dirichlet boundary conditions in the upper-half cylinder $Q_1^+\subset \R^{n+1}$, where the coefficients are weighted by $x_n^\al$, $\al\in(-\infty,1)$. We establish higher-order boundary Schauder type estimates of $x_n^\al u$ under the assumption that the coefficients have partially Dini mean oscillation. As an application, we also achieve higher-order boundary Harnack principles for degenerate or singular equations with Hölder continuous coefficients.
\end{abstract}

\maketitle
%\tableofcontents

\section{Introduction}

\subsection{Degenerate or singular parabolic systems}
For $\al\in(-\infty,1)$ and $Q_1^+=(-1,0)\times B_1'\times(0,1)\subset\R^{n+1}$, $n\ge2$,  we consider a parabolic system in divergence form with a weight $x_n^\al$ in $Q_1^+$:
   \begin{align}
        \label{eq:pde}
        \begin{cases}
            D_\g(x_n^\al A^{\g\de}D_\de u)-x_n^\al \partial_tu=\div\bg&\text{in }Q_1^+,\\
            u=0&\text{on }Q'_1.
        \end{cases}
    \end{align}
Here, $u=(u^1,\ldots,u^m)^{\tr}$, $m\ge1$, is a (column) vector-valued function. The coefficients $A^{\g\de}=(a^{\g\de}_{ij})^m_{i,j=1}$ are $m\times m$ matrices, $1\le\g,\de\le n$, and satisfy for some constant $\la>0$,
\begin{align}\label{eq:assump-coeffi}
|A^{\g\de}(X)|\le 1/\la,\quad \la|\xi|^2\le a^{\g\de}_{ij}(X)\xi_i^\g\xi_j^\de
\end{align}
for any $X\in Q_1^+$ and $\xi=(\xi_i^\g)\in \R^{m\times n}$. We sometimes denote $A^{\g\de}$ by $A$ for abbreviation.

Equations and systems with singular or degenerate coefficients arise widely in both pure and applied mathematics. For instance, solutions of equations involving fractional Laplacian can be seen as solutions of singular or degenerate equations via the Caffarelli-Silvestre extension \cite{CafSil07}. In finance, mathematical models derived from phenomena involving multiple financial instruments often take the form of degenerate or singular equations, such as the Black–Scholes equation \cite{BlaSch73}. In particular, multi-factor models that incorporate variables like stock price, volatility, and interest rates require systems rather than scalar equations. Degenerate or singular equations are also closely related to probability theory—for instance, the Fokker–Planck equation, which describes the time evolution of probability densities, inherently involves degenerate or singular coefficients \cite{Pav14}.

The main objective of this paper is to establish Schauder type estimates for solutions of \eqref{eq:pde} when the coefficients are of partially Dini mean oscillation.

\begin{definition}\label{def:partial-DMO} For $\theta\in(-1,\infty)$, we let $d\mu^\theta(X):=x_n^{\theta}dX$. We say that a measurable function $f$ in $Q_1^+$ is of \emph{partially Dini mean oscillation with respect to $X'$}, denoted as $L^1(d\mu^\theta)\text{-DMO}_{X'}$, in $Q_1^+$ if the function $\eta_f:(0,1)\to(0,\infty)$ defined by
    \begin{align}\label{eq:pDMO-def}
    \eta_f^\theta(r):=\sup_{X_0\in Q_1^+}\dashint_{Q_r^+(X_0)}\left|f(X)-\dashint_{Q_r'(X_0')}f(Y',x_n)dY'\right|d\mu^\theta(X)
    \end{align}
    is a Dini function, i.e., 
    $$
    \int_0^1\frac{\eta_f^\theta(r)}rdr<\infty.
    $$
\end{definition}

A related problem was considered in \cite{DonJeo24}, but under conormal boundary conditions rather than Dirichlet boundary conditions. In the conormal boundary case, higher regularity of $u$ was obtained. However, for our Dirichlet boundary setting, as observed in \cite{DonPha21}, one expects the regularity for $\uu:=x_n^\al u$, rather than $u$ itself.

To prove this result, we first obtain Schauder type estimates for the parabolic system \eqref{eq:pde} under the assumption that the coefficients belong to Hölder space; see Thoerem~\ref{thm:higher-reg-Holder}. To the best of knowledge, this result is new even for elliptic equations.

At the technical level, we follow the approach in \cite{DonJeo24} and apply  Campanato's method in a neat way. However, it is worth noting that the technical aspects form the central and most intricate challenges.

\subsection{Boundary Harnack Principle}
The higher-order Boundary Harnack principle concerns the higher regularity of the ratio of two scalar functions up to a part of the boundary where both functions vanish. In \cite{TerTorVit22}, the authors established Boundary Harnack principle for the uniformly elliptic equation with non-zero right-hand side by reducing the problem to Schauder type estimates of a certain degenerate equation. They considered divergence form equations with Hölder type coefficients and right-hand sides. 

This result was later extended to the parabolic setting in \cites{AudFioVit24a,AudFioVit24b, DonJeo24}. We also refer to \cite{Kuk22} for related results concerning non-divergence form equation in the parabolic setting, and to \cites{JeoVit24, DonJeoVit23} for elliptic equations with Dini continuous and DMO type coefficients. 

In this paper, we establish a higher-order boundary Harnack principle for parabolic `degenerate or singular' equations with Hölder type coefficients and right-hand sides. For the proof, we follow the approach in \cite{DonJeo24}, the corresponding result for the uniformly parabolic equation. We show that the ratio of two solutions satisfy a certain degenerate equation. To derive the desired estimates, we apply Schauder type estimate for degenerate equations under both conormal and Dirichlet boundary conditions.

\subsection{Main results}
In this section, we state our main results. We use notation from Section~\ref{subsec:notation}. Throughout this paper, we fix $-\infty<\al<1$, and write
$$
\uu:=x_n^\al u.
$$

\begin{theorem}\label{thm:main-reg}
    Let $k\in\mathbb{N}$ and $u$ be a solution of \eqref{eq:pde}. If $A$ satisfies \eqref{eq:assump-coeffi} and $A,\bg\in \cH^{k-1}_{-\al^+}(\overline{Q_1^+})$, then $\uu\in \mathring{C}^{k/2,k}(\overline{Q_{1/2}^+})$ and $x_n^\al Du\in \mathring{C}^{\frac{k-1}2,k-1}(\overline{Q_{1/2}^+})$.
\end{theorem}

Once we have Theorem~\ref{thm:main-reg}, we can repeat the argument in \cite{DonJeo24}*{Remark~1.3} to obtain the following result:

\begin{remark}
    In Theorem~\ref{thm:main-reg}, $D_x^i\partial_t^iu_\al$, $i+2j=k$, are continuous in $X$ except $D_n^ku$.
\end{remark}

In Theorem~\ref{thm:main-reg}, derivatives of $\uu\in \mathring{C}^{k/2,k}$ and $x_n^\al Du$ have moduli of continuity comparable with
\begin{align}\label{eq:mod-conti-sigma}\begin{split}
    \sigma_k(r):=&\left(\|\uu\|_{L^1(Q_1^+,d\mu_1)}+\|\bg\|_{L^\infty(Q_1^+)}+\int_0^1\frac{\eta_{\bg,k}(\rho)}\rho d\rho\right)\times \\
    &\quad\times\left(\int_0^r\frac{\eta_{A,k}(\rho)}{\rho}d\rho+r^\be\int_r^1\frac{\eta_{A,k}(\rho)}{\rho^{1+\be}}d\rho+r^\be
 \right)\\
 &\qquad+\int_0^r\frac{\eta_{\bg,k}(\rho)}\rho d\rho+r^\be\int_r^1\frac{\eta_{\bg,k}(\rho)}{\rho^{1+\be}}d\rho
\end{split}\end{align}
for any chosen $0<\be<1$, where $d\mu_1=x_n^{-\al^+}dX$. More precisely, if $i+2j=k$ for some $i,j\in \mathbb{Z}_+$, then
$$
|D_x^i\partial_t^j\uu(t,x',x_n)-D_x^i\partial_t^j\uu(s,y',x_n)|\lesssim\sigma_k(|x'-y'|+|t-s|^{1/2}).
$$
Moreover, if $i+2j=k-1$ for some $i,j\in \mathbb{Z}_+$ and $0<h<1/4$, then
\begin{align*}
    \left|\frac{\delta_{t,h}D_x^i\partial_t^j\uu}{h^{1/2}}(t,x)-\frac{\delta_{t,h}D_x^j\partial_t^j\uu}{h^{1/2}}(s,y)\right|\lesssim \sigma_k(|x-y|+|t-s|^{1/2}).
\end{align*}

In the theorem, we establish the regularity of both $\uu$ and $x_n^\al Du$. This formulation is essential for carrying out the induction argument needed to obtain the result for any $k\in\mathbb{N}$.

As mentioned earlier, while we follow the approach in \cite{DonJeo24} to prove Theorem~\ref{thm:main-reg}, the proof is much more technical in our setting. The key steps are deriving growth estimates of appropriate quantities $\psi$ and $\vp$ in \eqref{eq:psi-def} and \eqref{eq:phi-def}, respectively; compare with their counterparts in \cite{DonJeo24}.

\medskip

Our next result concerns the boundary Harnack principle for degenerate or singular parabolic equations.

\begin{theorem}
    \label{thm:par-BHP-deg}
    Let $k\in\mathbb{N}$, $-\infty<\al<1$ and $0<\be<1$. Assume that two functions $u$ and $v$ solve
    \begin{align*}\begin{cases}
        \div(x_n^\al A\D u)-x_n^\al\partial_tu=f,\quad u>0&\text{in }Q_1^+,\\
        \div(x_n^\al A\D v)-x_n^\al\partial_tv=g&\text{in }Q_1^+,\\
        u=v=0,\quad \partial_n(x_n^\al u)>0&\text{on }Q_1'.
    \end{cases}\end{align*}
Suppose that $A$ is symmetric and satisfies \eqref{eq:assump-coeffi} and $A,f,g\in C^{\frac{k-1+\be}2,k-1+\be}(\overline{Q_1^+})$. Then we have $v/u\in C^{\frac{k+\be}2,k+\be}(\overline{Q^+_{1/2}})$.
\end{theorem}

Due to some technical difficulties, we prove the parabolic boundary Harnack principle for the case with Hölder-type coefficients instead of the (partially) DMO-type coefficients. In \cite{DonJeo24}, the authors established Theorem~\ref{thm:par-BHP-deg} when $\al=0$ by reducing the problem to a degenerate equation:
$$
\div(x_n^2A\D w)-x_n^2\partial_tw=x_n(\tilde f+\mean{\tilde\bb,\D w})\quad\text{for some }\tilde f,\,\tilde\bb.
$$
In our case with general $\al\in(-\infty,1)$, we reduce to
$$
\div(x_n^{2-\al}A\D w)-x_n^{2-\al}\partial_tw=x_n^{1-\al}(\bar f+\mean{\bar\bb,\D w})\quad\text{for some }\bar f,\,\bar\bb.
$$

\subsection{Notation and structure of the paper}\label{subsec:notation}
Throughout the paper, we shall use $X=(t,x',x_n)=(t,x)=(X',x_n)$ to denote a point in $\R^{n+1}$, where $x'=(x_1,\ldots,x_{n-1})$, $x=(x',x_n)$, and $X'=(t,x')$. Similarly, we also write $Y=(s,y)=(s,y',y_n)=(Y',y_n)$ and $X_0=(t_0,x_0)=(t_0,x_0',(x_0)_n)=(X_0',(x_0)_n)$, etc.

For $X\in\R^{n+1}$, we denote
\begin{align*}
    &B'_r(x'):=\{y'\in\R^{n-1}\,:\, |y'-x'|<r\},\quad D_r(x):=B'_r(x')\times(x_n-r,x_n+r),\\
    &D_r^+(x):=B'_r(x')\times((x_n-r)^+,x_n+r),\quad Q'_r(X'):=(t-r^2,t)\times B'_r(x'),\\
    &Q_r(X):=(t-r^2,t)\times D_r(x),\quad Q_r^+(X):=(t-r^2,t)\times D^+_r(x),\\
&\mathbb{Q}_r(X):=(t-r^2,t+r^2)\times D_r(x).
\end{align*}
For abbreviation, we simply write $D_r=D_r(0)$, $D_r^+=D_r^+(0)$, $Q'_r=Q'_r(0)$, etc.

For $\al\in\R$, we write $\al^+:=\max\{\al,0\}\ge0$ and $\al^-:=\min\{\al,0\}\le0$.

Let $\mathbb{Z}_+=\mathbb{N}\cup\{0\}$ be the set of nonnegative integers.

Given $f:Q_1^+\to\R^m$, $h\in(0,1)$, and $1\le\g\le n$, define
\begin{align*}
    &\de_{t,h}f(t,x):=f(t,x)-f(t-h,x).
\end{align*}

Let $\cD\subset\R^n$ be a bounded open set in $\R^n$ and $\cQ:=(a,b)\times\cD$ for some $a,b\in\R$. We define the parabolic H\"older classes $C^{l/2,l}$, for $l=k+\be$, $k\in \mathbb{Z}_+$, $0<\be\le1$, as follows. We let
\begin{align*}
    &[u]_\cQ^{(k)}:=\sum_{|\mathbf\al|+2j=k}\|D_x^{\mathbf\al}\partial_t^ju\|_{L^\infty(\cQ)},\\
    &[u]_{x,\cQ}^{(\be)}:=\sup_{(t,x),(t,y)\in\cQ}\frac{|u(t,x)-u(t,y)|}{|x-y|^\be},\quad [u]_{t,\cQ}^{(\be)}:=\sup_{(t,x),(s,x)\in\cQ}\frac{|u(t,x)-u(s,x)|}{|t-s|^\be},\\
    &[u]_\cQ^{(l)}:=\sum_{|\mathbf\al|+2j=k}[D_x^{\mathbf\al}\partial_t^ju]_{x,\cQ}^{(\be)}+\sum_{k-1\le|\mathbf\al|+2j\le k}[D_x^{\mathbf\al}\partial_t^ju]_{t,\cQ}^{\left(\frac{l-|\mathbf{\al}|-2j}2\right)}.
\end{align*}
Then, we denote $C^{l/2,l}(\cQ)$ to be the space of functions $u$ for which the following norm is finite:
$$
\|u\|_{C^{l/2,l}(\cQ)}:=\sum_{i=0}^k[u]_\cQ^{(i)}+[u]_\cQ^{(l)}.
$$

By $C_x^\be$ and $C_t^\be$, we mean the space of Hölder continuous functions with exponent $\be$ in space and time, respectively. Let $C_X$ and $C_{X'}$ be the spaces of continuous functions in $X$ and $X'$, respectively.

Also, for $\al<1$ and $m\in\mathbb{N}$, we let $H^{1,-\al}(\cQ;\R^m)$ be the weighted Sobolev space consisting of measurable functions $u:\cQ\to\R^m$ equipped with the norm
\begin{align*}
    \|u\|_{H^{1,-\al}(\cQ;\R^m)}&=\|u\|_{L^2(\cQ,x_n^{-\al}dX;\R^m)}+\|Du\|_{L^2(\cQ,x_n^{-\al}dX;\R^m)}\\
    &=\left(\int_\cQ|u(X)|^2x_n^{-\al}dX\right)^{1/2}+\left(\int_\cQ|Du(X)|^2x_n^{-\al}dX\right)^{1/2}.
\end{align*}

Recall the definition of $\eta_f^\theta$ in \eqref{eq:pDMO-def}. Given $k\in\mathbb{Z}_+$, we say that $A\in \cH^k_\theta(\overline{\cQ})$ if
\begin{align}\label{eq:coeffi-space}
    \begin{cases}
        \text{- $A\in L^\infty(\cQ)$ when $k=0$, while $A\in C^{k/2,k}(\overline{\cQ})$ when $k\ge1$,}\\
        \text{- if $i+2j=k$ for some $i,j\in \mathbb{Z}_+$, then $D_x^i\partial_t^jA$ is of $L^1(d\mu^\theta)\text{-DMO}_{X'}$ in $\cQ$,}\\
        \text{- if $i+2j=k-1$ for some $i,j\in \mathbb{Z}_+$, then for any $0<h<1/4$}\\
        \quad\text{$\frac{\delta_{t,h}D_x^i\partial_t^jA}{h^{1/2}}$ is of $L^1(d\mu^\theta)\text{-DMO}_{X'}$,}\\
        \text{- There is a Dini function $\eta^\theta_{A,k}$ such that $\eta^\theta_f\le \eta^\theta_{A,k}$ for every function}\\
        \qquad \text{$f=D_x^i\partial_t^jA,\,\, \frac{\delta_{t,h}D_x^i\partial_t^jA}{h^{1/2}}$ as above.}
    \end{cases}
\end{align}
Next, for given $k\in\mathbb{Z}_+$, we say that $v\in\mathring{C}^{k/2,k}(\overline{\cQ})$ if the following holds:
\begin{align*}
    \begin{cases}
       \text{- $v\in L^\infty(\cQ)$ when $k=0$, while $v\in C^{k/2,k}(\overline{\cQ})$ when $k\ge1$,}\\
        \text{- if $i+2j=k$ for some $i,j\in \mathbb{Z}_+$, then $D_x^i\partial_t^jv\in C_{X'}(\cQ)$,}\\
        \text{- if $i+2j=k-1$ for some $i,j\in \mathbb{Z}_+$, then $\frac{\delta_{t,h}D_x^i\partial_t^jv}{h^{1/2}}\in C_X(\cQ)$ uniformly in}\\
        \qquad\text{$h\in (0,1/4)$, and $\frac{\delta_{t,h}D_x^i\partial_t^jv}{h^{1/2}}(t,x)\to0$ as $h\to0$ uniformly in $(t,x)\in\cQ$.}
    \end{cases}
\end{align*}
We emphasize that in the second condition, the derivatives $D_x^i\partial_t^jv$ may not be continuous in $x_n$. When $k=0$, we simply write $\cH^0_\theta=\cH_\theta$ and $\mathring{C}^{0,0}=\mathring{C}$.

The relation $A\lesssim B$ is understood that $A\le CB$ for some constant $C>0$. $A\approx B$ means $A\lesssim B$ and $B\lesssim A$.

\medskip
The remainder of the paper is organized as follows. In section~\ref{sec:prel}, we provide some preliminary lemmas. Section~\ref{sec:Schauder} is devoted to the proof of Schauder type estimates, Theorem~\ref{thm:main-reg}, in the case $k=1$. We extend this result to general $k\in\mathbb{N}$ in Section~\ref{sec:HO}. Finally, in Section~\ref{sec:BHP}, we prove the boundary Harnack principle, Theorem~\ref{thm:par-BHP-deg}.

%%%%%%%%%%%%%%%%%%%%%%%%%%%%%%%%%%%%%%%%%%%%%%%%%%%

\section{Preliminary lemmas}\label{sec:prel}

The following results correspond to Lemmas 2.1, 2.2 and A.1 in \cite{DonJeo24}, respectively.

\begin{lemma}\label{lem:product-par-DMO-elliptic}
     For $\theta>-1$, if $f$ and $g$ are bounded and of $L^1(x_n^\theta dX)$-$\DMO_{X'}$ in $Q_1^+$, then $fg$ is of $L^1(x_n^\theta dX)$-$\DMO_{X'}$ in $Q_{1/2}^+$.
\end{lemma}

\begin{lemma}
    \label{lem:partial-DMO-weight-rela}
    Let $\theta_1\ge\theta_2>-1$. If $f$ is of $L^1(x_n^{\theta_2} dX)$-$\DMO_{X'}$ in $Q_1^+$, then it is of $L^1(x_n^{\theta_1} dX)$-$\DMO_{X'}$ in $Q^+_{1/2}$. Moreover,
    $$
    \eta_f^{\theta_1}(r)\le C(n,\theta_1,\theta_2)\eta_f^{\theta_2}(r).
    $$
\end{lemma}

\begin{lemma}
    \label{lem:appen}
    Let $\sigma:(0,1)\to[0,\infty)$ be a Dini function, $0<\ka<1$ and $0<\be<1$. For $\be'=\frac{1+\be}2$, we set
    $$
    \tilde\sigma(r):=\sum_{j=0}^\infty\ka^{\be' j}\sigma(\ka^{-j}r)[\ka^{-j}r\le1],\quad \hat\sigma(r):=\sup_{\rho\in[r,1]}(r/\rho)^{\be'}\tilde\sigma(r),\qquad 0<r<1,
    $$
    where we used Iverson's bracket notation (i.e., $[P]=1$ when $P$ is true, while $[P]=0$ otherwise). Then
    \begin{align*}
            \int_0^r\frac{\hat\sigma(\rho)}\rho d\rho\lesssim \int_0^r\frac{\sigma(\rho)}\rho d\rho+r^\be\int_r^1\frac{\sigma(\rho)}{\rho^{1+\be}}d\rho.
    \end{align*}
\end{lemma}

The following lemma is the weak type-(1,1) estimate, which will play a crucial role in the proof of Schauder type estimates in Section~\ref{sec:Schauder}.

\begin{lemma}
    \label{lem:weak-type-(1,1)-Diri}
    Suppose $d\mu_1=x_n^{-\al}dX$ for some $\al<1$. Given $X_0\in Q_1^+$ and $r\in(0,1)$, we let $\cD_r(x_0)$ and $\tilde \cD_r(x_0)$ be smooth domains in $\R^n_+$ with $D_r^+(x_0)\subset \cD_r(x_0)\subset D^+_{\frac43r}(x_0)$ and $D^+_{\frac32r}(x_0)\subset \tilde \cD_r(x_0)\subset D_{2r}^+(x_0)$, and set $\cQ_r(X_0):=(t_0-r^2,t_0)\times \cD_r(x_0)$ and $\tilde\cQ_r(X_0):=(t_0-4r^2,t_0)\times\tilde\cD_r(x_0)$. For $\bar A ^{\g\de} =\bar A^{\g\de}(x_n)$ satisfying \eqref{eq:assump-coeffi} in $\tilde \cQ_r(X_0)$ and $\bff\in H^{1,-\al}(\tilde\cQ_r(X_0);\R^{m\times n})$, let $u\in H^{1,-\al}(\tilde\cQ_r(X_0);\R^m)$ be a solution of
    \begin{align}
        \label{eq:hom-Diri}
        \begin{cases}
            D_\g(x_n^\al\bar A^{\g\de}D_\de u)-x_n^\al\partial_tu=\div(\bff\chi_{\cQ_r(X_0)})&\text{in }\tilde\cQ_r(X_0),\\
            u=0&\text{on }\partial_p\tilde\cQ_r(X_0).
        \end{cases}
    \end{align}
    Then there is a constant $C=C(n,\la,\al)>0$ such that for every $\tau>0$,
    \begin{align}
        &\mu_1(\{X\in\cQ_r(X_0)\,:\, |x_n^\al\D u(X)|>\tau\})\le \frac{C}\tau\int_{\cQ_r(X_0)}|\bff|d\mu_1,\label{eq:weak-type-1}\\
        &\mu_1(\{X\in\cQ_r(X_0)\,:\, |\uu(X)|>\tau\})\le \frac{Cr}\tau\int_{\cQ_r(X_0)}|\bff|d\mu_1.\label{eq:weak-type-2}
    \end{align}
\end{lemma}

\begin{proof}
For the proof, we follow the idea in \cite{DonJeo24}*{Lemmas 2.3 and 2.4}. We only prove \eqref{eq:weak-type-1} since \eqref{eq:weak-type-2} can be derived in a similar way as outlined in \cite{DonJeo24}*{Lemma~2.4}.

To verify \eqref{eq:weak-type-1}, we assume without loss of generality $t_0=0$. Given $\hat\bff=(\hat f_1,\ldots,\hat f_n)\in H^{1,-\al}(\tilde\cQ_r(X_0);\R^{m\times n})$, we define $\bff=(f_1,\ldots, f_n)$ by
\begin{align}\label{eq:f}
    \begin{cases}
        f^j_\g:=\hat f_\g^j+\bar a^{n\g}_{ji}\hat f_n^i&\text{when }1\le\g\le n-1,\,1\le j\le m,\\
        f^j_n:=\bar a^{nn}_{ji}\hat f_n^i&\text{when }1\le j\le m,
    \end{cases}
\end{align}
 so that for any $v$,
\begin{align}
    \label{eq:v-f}
    \mean{x_n^\al\D v,\bff}=\mean{(x_n^\al\D_{x'}v,V),\hat\bff}, \quad\text{where }V=x_n^\al(\bar A^{n\de})^{\tr}D_\de v.
\end{align}
Then a map $T:\hat\bff\longmapsto x_n^\al \D u$, where $u$ is the unique solution of \eqref{eq:hom-Diri} with $\bff$ as in \eqref{eq:f} is well defined and is a bounded linear operator on $L^2(\cQ_r(X_0);\R^{m\times n},d\mu_1)$ by \cite{DonPha21}*{Theorem~2.2}.

Let $Y_0\in \cQ_r(X_0)$ and $\rho\in(0,r/2)$ and assume that $\hat\bff\in L^2(\cQ_r(X_0);\R^{m\times n},d\mu_1)$ is supported in $Q_\rho(Y_0)\cap \cQ_r(X_0)$ and satisfies $\int_{Q_\rho(Y_0)\cap \cQ_r(X_0)}\hat\bff d\mu_1=0$. We claim that
\begin{align}
    \label{eq:weak-type-cond}
    \int_{\cQ_r(X_0)\setminus \bbQ_{2\rho}(Y_0)}|x_n^\al\D u|d\mu_1\le C\int_{Q_\rho(Y_0)\cap \cQ_r(X_0)}|\hat\bff|d\mu_1.
\end{align}
Once we have \eqref{eq:weak-type-cond}, then \eqref{eq:weak-type-1} follows by \cite{DonJeo24}*{Lemma~2.3}. To obtain \eqref{eq:weak-type-cond}, for $R\ge2\rho$ with $\cQ_r(X_0)\setminus \bbQ_R(Y_0)\neq\emptyset$ and a function $\bh\in C_c^\infty((\bbQ_{2R}(Y_0)\setminus \bbQ_R(Y_0))\cap \cQ_r(X_0);\R^{m\times n})$, we let $v\in H^{1,-\al}((-4r^2,0)\times\tilde\cD_r(x_0);\R^m)$ be a solution of 
\begin{align}
    \label{eq:weak-sol-v}
    \begin{cases}
        D_\g(x_n^\al(\bar A^{\de\g})^{\tr}D_\de v)+x_n^\al\partial_tv=\div\bh\qquad\text{in }(-4r^2,0)\times\tilde\cD_r(x_0),\\
        v=0\qquad\text{on }((-4r^2,0]\times\partial\tilde\cD_r(x_0))\cup(\{0\}\times\tilde\cD_r(x_0)).
    \end{cases}
\end{align}
From the definition of weak solutions of \eqref{eq:hom-Diri} and \eqref{eq:weak-sol-v}, the equality \eqref{eq:v-f}, and the assumption on $\hat\bff$, we infer
\begin{align*}
    &\int_{\cQ_r(X_0)}\mean{x_n^\al\D u,\bh}d\mu_1=\int_{\cQ_r(X_0)}\mean{x_n^\al \D v,\bff}d\mu_1\\
    &=\int_{\cQ_r(X_0)}\mean{(x_n^\al\D_{x'}v,V),\hat\bff}d\mu_1=\int_{Q_\rho(Y_0)\cap \cQ_r(X_0)}\mean{(x_n^\al\D_{x'}v,V),\hat\bff}d\mu_1\\
    &=\int_{Q_\rho(Y_0)\cap \cQ_r(X_0)}\mean{(x_n^\al\D_{x'}v-\mean{x_n^\al\D_{x'}v}_{Q_\rho(Y_0)\cap \cQ_r(X_0)},V-\mean{V}_{Q_\rho(Y_0)\cap \cQ_r(X_0)}),\hat\bff}d\mu_1.
\end{align*}
On the other hand, since $v$ satisfies the homogeneous equation $D_\g(x_n^\al(\bar A^{\de\g})^{\text{tr}}D_\de v)+x_n^\al\partial_tv=0$ in $\bbQ_R(Y_0)\cap\cQ_r(X_0)$ with $v=0$ on $\partial(\bbQ_R(Y_0)\cap\cQ_r(X_0))\cap\{x_n=0\}$, we have by \cite{DonPha21}*{(B.1)} that
\begin{align*}
    [(x_n^\al\D_{x'}v,V)]_{C^{1/2,1}(Q_{R/2}(Y_0)\cap \cQ_r(X_0))}\le \frac{C}{R}\left(\dashint_{Q_R(Y_0)\cap\tilde\cQ_r(X_0)}|x_n^\al\D v|^2d\mu_1\right)^{1/2}.
\end{align*}
By using the previous equality and inequality, we can argue as in the proof of \cite{DonXu21}*{Lemma~4.3} to get \eqref{eq:weak-type-cond}, and complete the proof.
\end{proof}

\section{Schauder type estimates}\label{sec:Schauder}
In this section, we prove our central result, Theorem~\ref{thm:main-reg}, when $k=1$. We will extend this result to arbitrary $k\in\mathbb{N}$ in Section~\ref{sec:HO}.

\subsection{The regularity in space}\label{subsec:Schauder-space} The aim of this section is to prove the following result.

\begin{theorem}\label{thm:reg-Diri}
    For $\al<1$, let $u$ be a solution of \eqref{eq:pde}. Assume that $A$ satisfies \eqref{eq:assump-coeffi} and $A,\bg\in\cH_{-\al}(\overline{Q_1^+})$. Then $\uu$ is Lipschitz in $\overline{Q_{1/2}^+}$ with respect to the $x$-variable, and $x_n^\al \D_{x'}u$ and $x_n^\al A^{n\de}D_\de u-g_n$ are continuous in $\overline{Q_{1/2}^+}$ with respect to the $X$-variable.
\end{theorem}

In this section, we will derive an a priori estimate of the modulus of continuity of $x_n^\al\D_{x'}u$ and $x_n^\al A^{n\de}D_\de u-g_n$, assuming that $A$ and $\bg$ are smooth. Under this assumption, $x_n^\al \D u$ is bounded in $Q_{1/2}^+$. This can be shown by adapting the argument in the proof of \cite{DonPha21}*{Proposition~4.2}, which treats the homogeneous case with simple coefficients depending only on the $x_n$-variable. Although our setting involves more general coefficients and a nonzero right-hand side, the proof still applies since the Caccioppoli inequality and the difference quotient method in the $t$ and $x'$ variables remains valid. Then, the general case can be obtained by using a standard approximation argument.

In addition, in this section, we fix $\al<1$ and write for simplicity
\begin{align}\label{eq:not}
d\mu_1=x_n^{-\al}dX,\qquad \eta_A=\eta_A^{-\al},\qquad \eta_\bg=\eta_\bg^{-\al}.
\end{align}

We fix $0<p<1$, and for $X_0\in \overline{Q_{1/2}^+}$ and $0<r<1/2$, we set
\begin{align}\label{eq:psi-def}
    \psi(X_0,r):=\begin{cases}
        \inf_{\bq\in\R^{m\times n}}\left(\dashint_{Q_r(X_0)}|(d_{X_0}^\al\D_{x'}u,U)-\bq|^pd\mu_1\right)^{1/p},\quad 0<r\le d_{X_0}/2,\\
        \inf_{q\in\R^m}\left(\dashint_{Q_r^+(X_0)}|(x_n^\al\D_{x'}u, U-q)|^pd\mu_1\right)^{1/p},\quad d_{X_0}/2<r<1/2,
    \end{cases}
\end{align}
where $d_{X_0}:=(x_0)_n$ and $U:=x_n^\al A^{n\de}D_\de u-g_n$.

Next, we consider a few of Dini functions stemmed from $\eta_\bullet$, where $\bullet$ is either $A$ or $\bg$. For some constants $C>c>0$, depending only on $n$ and $\al$, we have
\begin{align}
    \label{eq:alm-mon}
    c\eta_\bullet(r)\le \eta_\bullet(s)\le C\eta_\bullet(r)
\end{align}
whenever $r/2\le s\le r<1$. See \cite{ChoDon19}*{(8.2)}. For a small constant $0<\ka<1$, we set
\begin{align}\label{eq:tilde-eta}
\tilde\eta_\bullet(r):=\sum_{i=0}^\infty\ka^{\be' i}\eta_\bullet(\ka^{-i}r)[\ka^{-i}r\le 1],\quad 0<r<1,
\end{align}
where we used Iverson's bracket notation (i.e., $[P]=1$ when $P$ is true, while $[P]=0$ otherwise). We also let
\begin{align}
    \label{eq:hat-eta}
    \hat\eta_\bullet(r):=\sup_{\rho\in [r,1)}(r/\rho)^{\beta'}\tilde\eta_\bullet(\rho),\quad 0<r<1.
\end{align}
Then $\tilde\eta_\bullet$ and $\hat\eta_\bullet$ are Dini functions satisfying \eqref{eq:alm-mon} and $\hat\eta_\bullet\ge\tilde\eta_\bullet\ge \eta_\bullet$. Moreover, $r\longmapsto \frac{\hat\eta_\bullet(r)}{r^{\beta'}}$ is nonincreasing. See e.g., \cite{Don12}.

An important step in the proof of Theorem~\ref{thm:reg-Diri} is the growth estimate of $\psi$, given in \eqref{eq:psi-est-Diri}. To derive this estimate, we first need to consider the boundary case (Lemma~\ref{lem:bdry-est-Diri}) and the interior case (Lemma~\ref{lem:int-est-Diri}).

\begin{lemma}
    \label{lem:bdry-est-Diri}
    Let $0<p<1$ and $0<\be<1$. If $\bar X_0\in Q'_{1/2}$, then we have for any $0<\rho<r<1/4$,
    $$
    \psi(\bar X_0,\rho)\le C(\rho/r)^{\be'}\psi(\bar X_0,r)+C\|x_n^\al\D u\|_{L^\infty(Q_{2r}^+(\bar X_0))}\tilde\eta_A(2\rho)+C\tilde\eta_\bg(2\rho),
    $$
    where $C=C(n,m,\la,\al,\be,p)>0$.
\end{lemma}

\begin{proof}
Without loss of generality, we may assume $\bar X_0=0$. We fix $0<r<1/4$ and denote
\begin{align*}
    &\bar A^{\g\de}(x_n):=\dashint_{Q'_{2r}}A^{\g\de}(Y',x_n)dY',\\
    &\bar\bg(x_n)=(\bar g_1(x_n),\ldots,\bar g_n(x_n)):=\dashint_{Q'_{2r}}\bg(Y',x_n)dY'.
\end{align*}
We observe that 
$$
D_\g(x_n^\al\bar A^{\g\de}D_\de u)-x_n^\al\partial_tu=D_\g(g_\g+x_n^\al(\bar A^{\g\de}-A^{\g\de})D_\de u)\quad\text{in }Q_1^+.
$$
We let 
$$
u_0(x_n):=\int_0^{x_n}s^{-\al}(\bar A^{nn}(s))^{-1}\bar g_n(s)\,ds\quad\text{in }Q^+_{2r}
$$
so that $u_e(X',x_n):=u(X',x_n)-u_0(x_n)$ satisfies
\begin{align*}
    \begin{cases}
        D_\g(x_n^\al\bar A^{\g\de}D_\de u_e)-x_n^\al\partial_tu_e=D_\g((g_\g-\bar g_\g+x_n^\al(\bar A^{\g\de}-A^{\g\de})D_\de u)\chi_{\cQ_r}) &\text{in }Q_{2r}^+,\\
        u_e=0&\text{on }Q'_{2r}.
    \end{cases}
\end{align*}
We take smooth domains $\cD_r$ and $\tilde\cD_r$ satisfying $D_r^+\subset \cD_r\subset D^+_{\frac43r}$ and $D^+_{\frac32r}\subset \tilde\cD_r\subset D^+_{2r}$, and write $\cQ_r:=\cD_r\times(-r^2,0)$ and $\tilde\cQ_r:=\tilde\cD_r\times(-4r^2,0)$. Let $w$ be a solution of
\begin{align*}
    \begin{cases}
        D_\g(x_n^\al\bar A^{\g\de}D_\de w)-x_n^\al\partial_tw=D_\g((g_\g-\bar g_\g+x_n^\al(\bar A^{\g\de}-A^{\g\de})D_\de u)\chi_{\cQ_r})&\text{in }\tilde\cQ_r,\\
        w=0&\text{on }\partial_p\tilde\cQ_r.
    \end{cases}
\end{align*}
Applying the weak type-(1,1) estimate \eqref{eq:weak-type-1} and following the argument as in obtaining \cite{DonJeoVit23}*{(2.8)} give
\begin{align}
    \label{eq:w-est}
    \begin{split}
        \left(\dashint_{Q_r^+}|x_n^\al\D w|^pd\mu_1\right)^{1/p}&\le C\dashint_{Q_{2r}^+}|g_\g-\bar g_\g+x_n^\al(\bar A^{\g\de}-A^{\g\de})D_\de u|d\mu_1\\
        &\le C\|x_n^\al\D u\|_{L^\infty(Q_{2r}^+)}\eta_A(2r)+C\eta_\bg(2r).
    \end{split}
\end{align}

Next, we consider $v:=u_e-w$, which satisfies
\begin{align}
    \label{eq:v-Diri}
    \begin{cases}
        D_\g(x_n^\al\bar A^{\g\de}D_\de v)-x_n^\al\partial_tv=0&\text{in }Q_r^+,\\
        v=0&\text{on }Q'_r.
    \end{cases}
\end{align}
Note that $\D_{x'}v$ also satisfies \eqref{eq:v-Diri}. For $V:=x_n^\al\bar A^{n\de}D_\de v$, we apply \cite{DonPha21}*{(4.9)} to $\D_{x'}v$ and \cite{DonPha21}*{(5.15)} to $v$ with a scaling and a standard iteration (see e.g., \cite{Gia93}*{pages 80-82}) to obtain
\begin{align*}
    &y_n^\al|\D_{x'}v(Y)|\le \frac{Cy_n}r\left(\dashint_{Q_r^+}|x_n^\al\D_{x'}v|^pd\mu_1\right)^{1/p},\quad Y\in Q_{r/2}^+,\\
    &[V]_{C^{1/2,1}(Q^+_{r/2})}\le\frac{C}r\left(\dashint_{Q_r^+}|x_n^\al\D v|^pd\mu_1\right)^{1/p}.
\end{align*}
It follows that for a small constant $\ka\in(0,1/2)$ to be determined later,
\begin{align*}
    \dashint_{Q_{\ka r}^+}|(x_n^\al\D_{x'}v,V-\mean{V}_{Q^+_{\ka r}})|^pd\mu_1&\le C\ka^p\dashint_{Q_r^+}|x_n^\al\D v|^pd\mu_1\\
    &\le C\ka^p\dashint_{Q_r^+}|(x_n^\al\D_{x'}v,V)|^pd\mu_1.
\end{align*}
For any constant vector $q\in \R^m$, we let $\tilde v(X):=v(X)-\int_0^{x_n}s^{-\al} (\bar A^{nn}(s))^{-1}q\,ds$ and $\tilde V:=x_n^\al\bar A^{n\de}D_\de\tilde v$. Then it is easily seen that $\tilde v$ satisfies \eqref{eq:v-Diri}, $D_{x'}\tilde v=D_{x'}v$, and $\tilde V=V-q$. Thus applying the previous inequality with $\tilde v$ yields
\begin{align*}
    \dashint_{Q_{\ka r}^+}|(x_n^\al\D_{x'}v,V-\mean{V}_{Q_{\ka r}^+})|^pd\mu_1\le C\ka^p\dashint_{Q_r^+}|(x_n^\al\D_{x'}v,V-q)|^pd\mu_1.
\end{align*}
On the other hand, by using $x_n^\al\bar A^{n\de}D_\de u_0-\bar g_n=0$, we get $U=V+x_n^\al\bar A^{n\de}D_\de w+\bar g_n-g_n$. This, along with the previous inequality and the equality $\D_{x'}u=\D_{x'}v+\D_{x'}w$, gives
\begin{align*}
    &\dashint_{Q_{\ka r}^+}|(x_n^\al\D_{x'}u,U-\mean{V}_{Q_{\ka r}^+})|^pd\mu_1\\
    &\le C\dashint_{Q_{\ka r}^+}|(x_n^\al\D_{x'}v,V-\mean{V}_{Q_{\ka r}^+})|^pd\mu_1+C\dashint_{Q_{\ka r}^+}\left(|x_n^\al\D w|^p+|\bar g_n-g_n|^p\right)d\mu_1\\
    &\le C\ka^p\dashint_{Q_r^+}|(x_n^\al\D_{x'}v,V-q)|^pd\mu_1+C\dashint_{Q_{\ka r}^+}\left(|x_n^\al\D w|^p+|\bar g_n-g_n|^p\right)d\mu_1\\
    &\le C\ka^p\dashint_{Q_r^+}|(x_n^\al\D_{x'}u,U-q)|^pd\mu_1+C\ka^{-(n+2-\al)}\dashint_{Q_{r}^+}\left(|x_n^\al\D w|^p+|\bar g_n-g_n|^p\right)d\mu_1.
\end{align*}
Since $q\in\R$ is arbitrary, this gives by \eqref{eq:w-est}
$$
\psi(0,\ka r)\le C\ka\psi(0,r)+C\ka^{-\frac{n+2-\al}p}\left(\|x_n^\al\D u\|_{L^\infty(Q^+_{2r})}\eta_A(2r)+\eta_\bg(2r)\right).
$$
We take $\ka\in(0,1/2)$ small so that $C\ka\le \ka^{\be'}$ and use a standard iteration to conclude Lemma~\ref{lem:bdry-est-Diri}. 
\end{proof}

\begin{lemma}
    \label{lem:int-est-Diri}
Let $p$, $\be$, and $\be'$ be as above. If $X_0\in Q_{1/2}^+$ and $0<\rho<r<d_{X_0}/2$, then
$$
\psi(X_0,\rho)\le C(\rho/r)^{\be'}\psi(X_0,r)+C\|x_n^\al\D u\|_{L^\infty(Q_r(X_0))}\tilde\eta_A(\rho)+C\tilde\eta_\bg(\rho).
$$
\end{lemma}

\begin{proof}
We simply write $d=d_{X_0}$ and note that $u$ solves
$$
D_\g(\tilde A^{\g\de}D_\de u)-a_0(x_n)\partial_tu=\div\tilde\bg\quad\text{in }Q_r(X_0),
$$
where $\tilde A^{\g\de}=(x_n/d)^\al A^{\g\de}$, $a_0(x_n)=(x_n/d)^\al$ and $\tilde\bg=(\tilde g_1,\ldots,\tilde g_n)=d^{-\al}\bg$. Since $\tilde A^{\g\de}$ is uniformly elliptic and $a_0(x_n)$ is between $(1/2)^\al$ and $(3/2)^\al$ and is smooth in $Q_r(X_0)$, by following the argument in the proof of \cite{DonXu21}*{Lemma~4.4}, we derive that for any $\ka\in(0,1/2)$ small
\begin{align*}
    &\psi^0(X_0,\ka r)\le C\ka\psi^0(X_0,r)+C\ka^{-\frac{n+2}p}\Bigg(\dashint_{Q_r(X_0)}\bigg|\tilde g_\g(X)-\dashint_{Q_r'(X_0')}\tilde g_\g(Y',x_n)dY'   \\
    &\qquad\qquad\qquad\qquad\quad\quad+\left(\dashint_{Q'_r(X_0')}\tilde A^{\g\de}(Y',x_n)dY'-\tilde A^{\g\de}(X)\right)D_\de u(X)\bigg|^pdX\Bigg)^{1/p},
\end{align*}
where
$$
\psi^0(X_0,r):=\inf_{\bq\in\R^{m\times n}}\left(\dashint_{Q_r(X_0)}|(\D_{x'}u,\tilde A^{n\de}D_\de u-\tilde g_n)-\bq|^pdX\right)^{1/p}.
$$
By multiplying $d^\al$ in both sides in the previous equality and using the fact that $(d/x_n)^{\al}$ is between $(1/2)^{-\al}$ and $(3/2)^{-\al}$ in $Q_r(X_0)$, we infer
\begin{align*}
    \psi(X_0,\ka r)\le C\ka\psi(X_0,r)+C\ka^{-\frac{n+2}p}\left(\|x_n^\al\D u\|_{L^\infty(Q_r(X_0))}\eta_A(r)+\eta_\bg(r)\right).
\end{align*}
By taking $\ka$ small and using a standard iteration, this implies Lemma~\ref{lem:int-est-Diri}.    
\end{proof}

In the rest of this section, when there is no confusion, we simply write $d=d_{X_0}$ and $\|\bg\|_\infty=\|\bg\|_{L^\infty(Q_1^+)}$.

\begin{lemma}
    \label{lem:psi-est-Diri}
Let $0<p<1$ and $0<\be<1$. Then we have for any $X_0\in Q_{1/2}^+$ and $0<\rho\le r<1/16$,
\begin{align}
    \label{eq:psi-est-Diri}\begin{split}
    \psi(X_0,\rho)&\le C(\rho/r)^{\be'}\left(\dashint_{Q_{8r}^+(X_0)}|x_n^\al\D u|d\mu_1+\|\bg\|_\infty\right)\\
    &\qquad+C\left(\|x_n^\al\D u\|_{L^\infty(Q_{8r}^+(X_0))}\hat\eta_A(\rho)+\hat\eta_\bg(\rho)\right).
\end{split}\end{align}
\end{lemma}

\begin{proof}
we consider three cases $\rho<r<d/2$, $d/2<\rho<r$ and $\rho\le d/2\le r$.

\medskip\noindent\emph{Case 1.} Suppose $\rho<r<d/2$. By using $|U|\le C|x_n^\al\D u|+|\bg|$ and $x_n/d\in(1/2,3/2)$ for every $X\in Q_{d/2}(X_0)$, we get
\begin{align}\label{eq:psi-bound}
    \psi(X_0,r)\le\left(\dashint_{Q_r(X_0)}|(x_n^\al\D_{x'}u,U)|^pd\mu_1\right)^{1/p}\le C\left(\dashint_{Q_r(X_0)}|x_n^\al\D u|d\mu_1+\|\bg\|_\infty\right).
\end{align}
Combining this with Lemma~\ref{lem:int-est-Diri} gives \eqref{eq:psi-est-Diri}.

\medskip\noindent\emph{Case 2.} Suppose $d/2<\rho<r$. Then, for $\bar X_0:=(X'_0,0)\in Q'_{1/2}$ we have
\begin{align*}
    \psi(\bar X_0,3r)\le\left(\dashint_{Q_{3r}^+(\bar X_0)}|(x_n^\al\D_{x'}u,U)|^pd\mu_1\right)^{1/p}\le C\left(\dashint_{Q_{5r}^+(X_0)}|x_n^\al\D u|d\mu_1+\|\bg\|_\infty\right).
\end{align*}
This, together with Lemma~\ref{lem:bdry-est-Diri} and inclusions $Q_\rho^+(X_0)\subset Q_{3\rho}^+(\bar X_0)\subset Q^+_{3r}(\bar X_0)\subset Q_{5r}^+(X_0)$, yields
\begin{align*}
    \psi(X_0,\rho)&\le C\psi(\bar X_0,3\rho)\\
    &\le C(\rho/r)^{\be'}\psi(\bar X_0,3r)+C\left(\|x_n^\al\D u\|_{L^\infty(Q_{6r}^+(\bar X_0))}\tilde\eta_A(6\rho)+\tilde\eta_\bg(6\rho)\right)\\
    &\le C(\rho/r)^{\be'}\left(\dashint_{Q_{5r}^+(X_0)}|x_n^\al\D u|d\mu_1+\|\bg\|_\infty\right)\\
    &\qquad+C\left(\|x_n^\al\D u\|_{L^\infty(Q_{6r}^+(\bar X_0))}\tilde\eta_A(6\rho)+\tilde\eta_\bg(6\rho)\right).
\end{align*}
Since $\tilde\eta_\bullet\le\hat\eta_\bullet$ and $t\longmapsto\frac{\hat\eta_\bullet(t)}{t^{\be'}}$ is nonincreasing, \eqref{eq:psi-est-Diri} follows.

\medskip\noindent\emph{Case 3.} We consider the last case $\rho\le d/2\le r$. We recall $x_n/d\in(1/2,3/2)$ for any $X\in Q_{d/2}(X_0)$ to have
\begin{align*}
    \psi(X_0,d/2)&\le C\left(\dashint_{Q_{d/2}(X_0)}|x_n^\al\D_{x'}u|^pd\mu_1\right)^{1/p}+C\inf_{q\in\R^m}\left(\dashint_{Q_{d/2}(X_0)}|U-q|^pd\mu_1\right)^{1/p}\\
    &\le C\psi(\bar X_0,3d/2).
\end{align*}
This, along with Lemmas~\ref{lem:bdry-est-Diri} and \ref{lem:int-est-Diri} and the monotonicity of $\rho\longmapsto\frac{\hat\eta_\bullet(\rho)}{\rho^{\be'}}$, gives
\begin{align*}
    \psi(X_0,\rho)&\le C(\rho/d)^{\be'}\psi(X_0,d/2)+C(\|x_n^\al\D u\|_{L^\infty(Q_{d/2}(X_0))}\tilde\eta_A(\rho)+\tilde\eta_\bg(\rho))\\
    &\le C(\rho/d)^{\be'}\psi(\bar X_0,3d/2)+C(\|x_n^\al\D u\|_{L^\infty(Q_{d/2}(X_0))}\tilde\eta_A(\rho)+\tilde\eta_\bg(\rho))\\
    &\le C(\rho/r)^{\be'}\psi(\bar X_0,3r)+C(\|x_n^\al\D u\|_{L^\infty(Q_{d/2}(X_0))}\tilde\eta_A(\rho)+\tilde\eta_\bg(\rho))\\
    &\qquad+C(\rho/d)^{\be'}\left(\|x_n^\al\D u\|_{L^\infty(Q^+_{6r}(\bar X_0))}\tilde\eta_A(3d)+\tilde\eta_\bg(3d)\right)\\
    &\le C(\rho/r)^{\be'}\left(\dashint_{Q_{5r}^+(X_0)}|x_n^\al\D u|d\mu_1+\|\bg\|_\infty\right)\\
    &\qquad+C\left(\|x_n^\al\D u\|_{L^\infty(Q^+_{8r}(X_0))}\hat\eta_A(\rho)+\hat\eta_\bg(\rho)\right).
\end{align*}
This completes the proof.
\end{proof}

For $X_0\in Q_{1/2}^+$, $d=d_{X_0}$ and $0<r<1/16$, we take a constant vector $\bq_{X_0,r}\in\R^{m\times n}$ such that
\begin{align}\label{eq:psi-q}
    \psi(X_0,r)=\begin{cases}\left(\dashint_{Q_r(X_0)}|(d^\al\D_{x'}u,U)-\bq_{X_0,r}|^pd\mu_1\right)^{1/p},&0<r\le d/2,\\
        \left(\dashint_{Q_r^+(X_0)}|(x_n^\al\D_{x'}u,U)-\bq_{X_0,r}|^pd\mu_1\right)^{1/p},&d/2<r<1/16,
    \end{cases}
\end{align}
where $\bq_{X_0,r}$ is of the form $(0,\ldots,0,q_{X_0,r})$ for some $q_{X_0,r}\in\R^m$ when $d/2<r<1/16$.

\begin{lemma}
    \label{lem:gradient-unif-est-Diri}
    It holds that
    \begin{align}\label{eq:gradient-unif-est-1-Diri}
        \|x_n^\al\D u\|_{L^\infty(Q_{1/2}^+)}\lesssim \int_{Q_1^+}|\uu|d\mu_1+\int_0^1\frac{\hat\eta_\bg(\rho)}\rho d\rho+\|\bg\|_{\infty}.
    \end{align}
\end{lemma}

\begin{proof}
We divide the proof into two steps.

\medskip\noindent\emph{Step 1.} In this first step, we show that for almost every $X_0\in Q_{1/2}^+$ and $0<r<1/20$,
\begin{align}
    \label{eq:gradient-unif-est-Diri}\begin{split}
    |(x_0)_n^\al\D u(X_0)|&\lesssim \dashint_{Q_{8r}^+(X_0)}|x_n^\al \D u|d\mu_1+\|\bg\|_\infty\\
    &\qquad+\|x_n^\al\D u\|_{L^\infty(Q_{10r}^+(X_0))}\int_0^r\frac{\hat\eta_A(\rho)}{\rho}d\rho+\int_0^r\frac{\hat\eta_\bg(\rho)}\rho d\rho.
\end{split}\end{align}
To prove \eqref{eq:gradient-unif-est-Diri}, we consider two cases either $0<r\le d/2$ or $d/2<r\le 1/20$. 

\medskip\noindent\emph{Case 1.} We first consider the case $r\le d/2$. By taking average of the trivial inequality
\begin{align*}
    |\bq_{X_0,r}-\bq_{X_0,r/2}|^p\le |(d^\al\D_{x'}u,U)-\bq_{X_0,r}|^p+|(d^\al\D_{x'}u,U)-\bq_{X_0,r/2}|^p
\end{align*}
over $Q_{r/2}(X_0)$ with respect to $\mu_1$ and taking the $p$th root, we get
$$
|\bq_{X_0,r}-\bq_{X_0,r/2}|\le C(\psi(X_0,r)+\psi(X_0,r/2)).
$$
By iteration, we have for any $k\in\mathbb{N}$,
\begin{align}\label{eq:q-est}
    |\bq_{X_0,2^{-k}r}-\bq_{X_0,r}|\le C\sum_{j=0}^k\psi(X_0,2^{-j}r).
\end{align}
On the other hand, from \eqref{eq:psi-est-Diri} we infer $\psi(X_0,2^{-k}r)\to0$ as $k\to\infty$. Then, by using the smoothness assumption on $A$ and $\bg$ and applying the Lebesgue differentiation theorem, we obtain that for almost every $X_0\in Q^+_{1/2}$,
$$
\lim_{k\to\infty}\bq_{X_0,2^{-k}r}=((x_0)_n^\al\D_{x'}u(X_0),U(X_0)).
$$
By combining this with \eqref{eq:q-est} and applying Lemma~\ref{lem:psi-est-Diri}, we get
\begin{align}
    \label{eq:u-q-diff-est}\begin{split}
    &|((x_0)_n^\al\D_{x'}u(X_0),U(X_0))-\bq_{X_0,r}|\\
    &\lesssim \left(\dashint_{Q_{8r}^+(X_0)}|x_n^\al\D u|d\mu_1+\|\bg\|_\infty\right)
    +\|x_n^\al\D u\|_{L^\infty(Q_{8r}^+(X_0))}\int_0^r\frac{\hat\eta_A(\rho)}\rho d\rho+\int_0^r\frac{\hat\eta_\bg(\rho)}\rho d\rho.
\end{split}\end{align}
To bound $|\bq_{X_0,r}|$, we observe that in $Q_r(X_0)$,
$$
|\bq_{X_0,r}|^p\le |(d^\al\D_{x'}u,U)-\bq_{X_0,r}|^p+C\left(|x_n^\al\D u|^p+|\bg|^p\right).
$$
This, along with \eqref{eq:psi-bound}, gives
\begin{align*}
    |\bq_{X_0,r}|&\le C\psi(X_0,r)+C\left(\dashint_{Q_r(X_0)}|x_n^\al\D u|^pd\mu_1\right)^{1/p}+C\|\bg\|_\infty\\
    &\le C\dashint_{Q_r(X_0)}|x_n^\al \D u|d\mu_1+C\|\bg\|_\infty.
\end{align*}
By combining this with \eqref{eq:u-q-diff-est}  and the inequality 
$$
|(x_0)_n^\al\D u(X_0)|\le C\left(|((x_0)_n^\al\D_{x'}u(X_0),U(X_0))|+\|\bg\|_\infty\right),
$$
we deduce \eqref{eq:gradient-unif-est-Diri}.

\medskip\noindent\emph{Case 2.} Suppose $d/2<r<1/20$. We take an integer $j_0\ge0$ such that $2^{-(j_0+1)}r\le d/2<2^{-j_0}r$. By arguing as in Case 1, we obtain that for any $j\ge0$ with $j\neq j_0$,
$$
|\bq_{X_0,2^{-j}r}-\bq_{X_0,2^{-(j+1)}r}|\le C\left(\psi(X_0,2^{-j}r)+\psi(X_0,2^{-(j+1)}r)\right).
$$
Notice that the case when $j=j_0$ is nontrivial due to the different structures of $\psi(X_0,r)$ depending on whether $r\le d/2$ or $r>d/2$. By iterating, we have for any $k\ge j_0+1$,
\begin{align}
    \label{eq:q-est-1}|\bq_{X_0,r}-\bq_{X_0,2^{-k}r}|\le |\bq_{X_0,2^{-j_0}r}-\bq_{X_0,2^{-(j_0+1)}r}|+C\sum_{j=0}^k\psi(X_0,2^{-j}r).
\end{align}
To estimate the first term in the right-hand side, we note that in $Q_{2^{-(j_0+1)}r}(X_0)$,
\begin{multline*}
    |\bq_{X_0,2^{-j_0}r}-\bq_{X_0,2^{-(j_0+1)}r}|^p\le |(d^\al\D_{x'}u,U)-\bq_{X_0,2^{-(j_0+1)}r}|^p\\
    +|(x_n^\al\D_{x'}u,U)-\bq_{X_0,2^{-j_0}r}|^p+|d^\al\D_{x'}u|^p+|x_n^\al\D_{x'}u|^p.
\end{multline*}
Since $x_n/d\in(1/2,3/2)$ in $Q_{2^{-(j_0+1)}r}(X_0)\subset Q_{d/2}(X_0)$, this inequality yields
\begin{align}
    \label{eq:q-est-2}\begin{split}
        |\bq_{X_0,2^{-j_0}r}-\bq_{X_0,2^{-(j_0+1)}r}|&\lesssim \psi(X_0,2^{-(j_0+1)}r)+\psi(X_0,2^{-j_0}r)\\
        &\qquad+\left(\dashint_{Q_{2^{-(j_0+1)}r}(X_0)}|x_n^\al\D_{x'}u|^pd\mu_1\right)^{1/p}.
    \end{split}
\end{align}
Moreover, from $2^{-(j_0+1)}r\le d/2<2^{-j_0}r$, we have $2^{-(j_0+1)}r+d\le \frac32d<2^{2-j_0}r$, thus $Q_{2^{-(j_0+1)}r}(X_0)\subset Q_{2^{2-j_0}r}(\bar X_0)$. Then we have by applying Lemma~\ref{lem:bdry-est-Diri} that
\begin{align*}
    &\left(\dashint_{Q_{2^{-(j_0+1)}r}(X_0)}|x_n^\al\D_{x'}u|^pd\mu_1\right)^{1/p}\\
    &\lesssim \left(\dashint_{Q_{2^{2-j_0}r}(\bar X_0)}|x_n^\al\D_{x'}u|^pd\mu_1\right)^{1/p}\lesssim \psi(\bar X_0,2^{2-j_0}r)\\
    &\lesssim \psi(\bar X_0,4r)+\|x_n^\al\D u\|_{L^\infty(Q^+_{8r}(\bar X_0))}\tilde\eta_A(2^{3-j_0}r)+\tilde\eta_\bg(2^{3-j_0}r)\\
    &\lesssim \psi(X_0,8r)+\|x_n^\al\D u\|_{L^\infty(Q^+_{8r}(\bar X_0))}\tilde\eta_A(2^{3-j_0}r)+\tilde\eta_\bg(2^{3-j_0}r).
\end{align*}
By combining this with \eqref{eq:q-est-1} and \eqref{eq:q-est-2} and applying Lemma~\ref{lem:psi-est-Diri},
\begin{align}
    \label{eq:q-est-3}\begin{split}
        &|\bq_{X_0,r}-\bq_{X_0,2^{-k}r}|\\
        &\lesssim \sum_{j=-3}^\infty\psi(X_0,2^{-j}r)+\|x_n^\al\D u\|_{L^\infty(Q_{8r}^+(\bar X_0))}\tilde\eta_A(2^{3-j_0}r)+\tilde\eta_\bg(2^{3-j_0}r)\\
        &\lesssim \dashint_{Q_{8r}^+(X_0)}|x_n^\al\D u|d\mu_1+\|\bg\|_\infty+\|x_n^\al\D u\|_{L^\infty(Q^+_{10r}(X_0))}\int_0^r\frac{\hat\eta_A(\rho)}\rho d\rho+\int_0^r\frac{\hat\eta_\bg(\rho)}\rho d\rho.
    \end{split}
\end{align}
Since the estimate $|\bq_{X_0,r}|\le C\left(\dashint_{Q_r^+(X_0)}|x_n^\al\D u|d\mu_1+\|\bg\|_\infty\right)$ can be derived by arguing as in Case 1, \eqref{eq:gradient-unif-est-Diri} follows by taking $k\to\infty$ in \eqref{eq:q-est-3}.

\medskip\noindent\emph{Step 2.} Once we have \eqref{eq:gradient-unif-est-Diri} obtained in Step 1, we can follow the same argument in Step 2 in the proof of \cite{DonJeoVit23}*{Lemma~2.8} to obtain
\begin{align}\label{eq:gradient-est-infty}
 \|x_n^\al\D u\|_{L^\infty(Q_{1/2}^+)}\lesssim \int_{Q_{3/4}^+}|x_n^\al \D u|d\mu_1+\int_0^1\frac{\hat\eta_\bg(\rho)}\rho d\rho+\|\bg\|_{\infty}.
\end{align}
In addition, by testing \eqref{eq:pde} by $u\xi^2$, where $\xi\in C^\infty(Q_1^+)$ satisfying $\xi=1$ in $Q^+_{1/2}$ and $\supp\xi\subset Q^+_{3/4}$, we get the Caccioppoli inequality
\begin{align}
    \label{eq:Caccio-ineq}
    \int_{Q_{1/2}^+}|x_n^\al\D u|^2d\mu_1\lesssim \int_{Q_{3/4}^+}|\uu|^2d\mu_1+\|\bg\|_\infty.
\end{align}
For any $X\in Q_{1/2}^+$, since $\al<1$, we have by \eqref{eq:gradient-est-infty} that
\begin{align}\label{eq:u-L^infty-bound}\begin{split}
    |\uu(X)|&=\left|x_n^{\al+1}\int_0^1D_nu(sx_n)\,ds \right|\lesssim x_n\int_0^1\|x_n^{\al}\D u\|_{L^\infty(Q_{1/2}^+)}s^{-\al}ds\\
    &\lesssim \|x_n^\al \D u\|_{L^\infty(Q_{1/2}^+)}x_n.
\end{split}\end{align}
By combining this with \eqref{eq:gradient-est-infty} and \eqref{eq:Caccio-ineq}, we get
\begin{align*}
    \|\uu\|_{L^\infty(Q_{1/2}^+)}\lesssim \|x_n^\al\D u\|_{L^\infty(Q^+_{1/2})}\lesssim \left(\dashint_{Q_{3/4}^+}|\uu|^2d\mu_1\right)^{1/2}+\int_0^1\frac{\hat\eta_\bg(\rho)}\rho d\rho+\|\bg\|_\infty.
\end{align*}
By using a standard iteration as in \cite{Gia83}*{pages 80-82}, we further have
\begin{align*}
    \|\uu\|_{L^\infty(Q_{1/2}^+)}\lesssim\dashint_{Q_{3/4}^+}|\uu|d\mu_1+\int_0^1\frac{\hat\eta_\bg(\rho)}\rho d\rho+\|\bg\|_\infty.
\end{align*}
This, combined with \eqref{eq:gradient-est-infty} and \eqref{eq:Caccio-ineq}, concludes \eqref{eq:gradient-unif-est-1-Diri}.
\end{proof}

\begin{remark}
    \label{rem:u-alpha-bound}
    Equation \eqref{eq:u-L^infty-bound} implies that $\uu=0$ on $Q'_{1/2}$, which is nontrivial when $\al<0$. Although we do not use this fact in the proof for Schauder type estimates, it will play a significant role in the proof for Boundary Harnack principles; see the proof of Theorem~\ref{thm:par-BHP-deg}.
\end{remark}

In the preceding lemma, we obtained the $L^\infty$-estimate for $x_n^\al\D u$. In the following, we extend this result to the $L^\infty$-estimate for $\D \uu$.

\begin{lemma}
    \label{lem:gradient-bound-Diri}
    There holds
    $$
    \|\D\uu\|_{L^\infty(Q_{1/2}^+)}\lesssim\int_{Q_1^+}|\uu|d\mu_1+\int_0^1\frac{\hat\eta_\bg(\rho)}\rho d\rho+\|\bg\|_{\infty}.
    $$
\end{lemma}

\begin{proof}
Lemma~\ref{lem:gradient-bound-Diri} follows from the identities $D_{x'}\uu=x_n^\al\D_{x'}u$ and $D_n\uu=x_n^\al D_nu+\al \uu/x_n$, the estimate \eqref{eq:u-L^infty-bound}, and Lemma~\ref{lem:gradient-unif-est-Diri}.
\end{proof}

For $0<\be<1$, we let $\omega_x:[0,1)\to[0,\infty)$ be a modulus of continuity defined by
\begin{align}\label{eq:omega_x}
    \omega_x(r):=\left(\|\uu\|_{L^1(Q_1^+,d\mu_1)}+\|\bg\|_\infty+\int_0^1\frac{\hat\eta_\bg(\rho)}\rho d\rho\right)\left(r^\be+\int_0^r\frac{\hat\eta_A(\rho)}\rho d\rho\right)+\int_0^r\frac{\hat\eta(\rho)}{\rho}d\rho.
\end{align}

Recall $\be'=\frac{\be+1}2$ and notice that the monotonicity of $\rho\longmapsto\frac{\hat\eta_\bullet(\rho)}{\rho^{\be'}}$ implies that of $\rho\longmapsto\frac{\omega_x(\rho)}{\rho^{\be'}}$. Moreover, for any $X_0\in Q^+_{1/2}$ and $0<r<1/16$, we have by Lemmas~\ref{lem:psi-est-Diri} and \ref{lem:gradient-unif-est-Diri} that
\begin{align}
    \label{eq:psi-omega}
    \psi(X_0,r)\le C\omega_x(r).
\end{align}

\begin{lemma}
    \label{lem:u-q-diff-est}
    If $X_0\in Q_{1/2}^+$ and $0<r<1/18$, then we have
    $$
    |((x_0)_n^\al\D_{x'}u(X_0),U(X_0))-\bq_{X_0,r}|\le C\omega_x(r),
    $$
    where $C=C(n,m,\la,\al,p,\be)>0$.
\end{lemma}

\begin{proof}
    The proof of Lemma~\ref{lem:u-q-diff-est} is similar to that of \cite{DonJeoVit23}*{Lemma~2.9}.
\end{proof}

Now we are ready to prove Theorem~\ref{thm:reg-Diri}. 

\begin{proof}[Proof of Theorem~\ref{thm:reg-Diri}]
Thanks to Lemma~\ref{lem:gradient-bound-Diri}, it is sufficient to prove the continuity of $x_n^\al\D_{x'}u$ and $U$. To this end, we will show that for any $X_0=(t_0,x_0)$ and $Y_0=(s_0,y_0)\in Q_{1/2}^+$ with $r:=|x_0-y_0|+\sqrt{|t_0-s_0|}>0$,
\begin{align}
    \label{eq:u-reg-Diri}
    |((x_0)_n^\al \D_{x'}u(X_0),U(X_0))-((y_0)_n^\al\D_{x'}u(Y_0),U(Y_0))|\le C\omega_x(r).
\end{align}
Indeed, due to Lemma~\ref{lem:gradient-unif-est-Diri}, we may assume $0<r<1/90$. Moreover, if $r\ge(x_0)_n/8$, then we can follow the argument in Case 1 in the proof of \cite{DonJeoVit23}*{Theorem~2.4} to derive \eqref{eq:u-reg-Diri}. Thus, we may also assume $r<(x_0)_n/8$. By Lemma~\ref{lem:u-q-diff-est} and the triangle inequality, we have
$$
|((x_0)_n^\al\D_{x'}u(X_0),U(X_0))-((y_0)_n^\al\D_{x'}u(Y_0),U(Y_0))|\le C\omega_x(2r)+|\bq_{X_0,2r}-\bq_{Y_0,2r}|.
$$
From $(y_0)_n>(x_0)_n-r>7r$, we see that $2r<\min\{(x_0)_n/2,(y_0)_n/2\}$. From the trivial inequality
\begin{align*}
    |\bq_{X_0,2r}-\bq_{Y_0,2r}|^p&\le |\bq_{X_0,2r}-(d_{X_0}^\al\D_{x'}u,U)|^p+|\bq_{Y_0,2r}-(d^\al_{Y_0}\D_{x'}u,U)|^p\\
    &\qquad+|(d_{X_0}^\al-d_{Y_0}^\al)\D_{x'}u|^p
\end{align*}
in $Q_r(X_0)$, we infer
\begin{align*}
    |\bq_{X_0,2r}-\bq_{Y_0,2r}|&\lesssim\psi(X_0,2r)+\psi(Y_0,2r)+\left(\dashint_{Q_r(X_0)}|(d_{X_0}^\al-d_{Y_0}^\al)\D_{x'}u|^pd\mu_1\right)^{1/p}\\
    &\lesssim \omega_x(r)+\left(\dashint_{Q_r(X_0)}|(d_{X_0}^\al-d_{Y_0}^\al)\D_{x'}u|^pd\mu_1\right)^{1/p},
\end{align*}
where we used \eqref{eq:psi-omega} in the second inequality. Thus, to obtain \eqref{eq:u-reg-Diri}, it suffices to show that
\begin{align}
    \label{eq:u-d-est-Diri}
    \left(\dashint_{Q_r(X_0)}|(d_{X_0}^\al-d_{Y_0}^\al)\D_{x'}u|^pd\mu_1\right)^{1/p}\le C\omega_x(r).
\end{align}
To prove it, note that $7/8<\frac{(y_0)_n}{(x_0)_n}<9/8$ as $(x_0)_n>8r$ and $|(x_0)_n-(y_0)_n|\le r$. Then we have in $Q_r(X_0)$,
$$
|(d_{X_0}^\al-d_{Y_0}^\al)\D_{x'}u|=\left|1-\left(\frac{(y_0)_n}{(x_0)_n}\right)^\al\right||d_{X_0}^\al\D_{x'}u|\le \frac{Cr}{(x_0)_n}|d_{X_0}^\al\D_{x'}u|,
$$
and thus
\begin{align}\label{eq:d-diff-est}
\left(\dashint_{Q_r(X_0)}|(d_{X_0}^\al-d_{Y_0}^\al)\D_{x'}u|^pd\mu_1\right)^{1/p}\le \frac{Cr}{(x_0)_n}\left(\dashint_{Q_r(X_0)}|d_{X_0}^\al\D_{x'}u|^pd\mu_1\right)^{1/p}.
\end{align}
To estimate the right-hand side in \eqref{eq:d-diff-est}, we write $d=d_{X_0}$ and take $k\in\mathbb{N}$ such that $d/8<2^kr\le d/4$. For each $0\le j\le k-1$, we note that $2^{j+1}r\le d/4$ and let $\bq_{X_0,2^{j+1}r}=(\bq'_{X_0,2^{j+1}r},(\bq_{X_0,2^{j+1}r})_n)\in\R^{m\times n}$ be as in \eqref{eq:psi-q}. Then
\begin{align*}
    &\dashint_{Q_{2^jr}(X_0)}|d^\al\D_{x'}u|^pd\mu_1\\
    &=\dashint_{Q_{2^jr}(X_0)}\dashint_{Q_{2^{j+1}r}(X_0)}|d^\al\D_{x'}u(Y)+(d^\al\D_{x'}u(X)-d^\al\D_{x'}u(Y))|^pd\mu_1(Y)d\mu_1(X)\\
    &\le \dashint_{Q_{2^jr}(X_0)}\dashint_{Q_{2^{j+1}r}(X_0)}|d^\al\D_{x'}u(Y)|^pd\mu_1(Y)d\mu_1(X)\\
    &\qquad+C\dashint_{Q_{2^{j+1}r}(X_0)}\dashint_{Q_{2^{j+1}r}(X_0)}|d^\al\D_{x'}u(X)-d^\al\D_{x'}u(Y)|^pd\mu_1(Y)d\mu_1(X)\\
    &\le \dashint_{Q_{2^{j+1}r}(X_0)}|d^\al\D_{x'}u|^pd\mu_1+C\dashint_{Q_{2^{j+1}r}(X_0)}|d^\al\D_{x'}u-\bq'_{X_0,2^{j+1}r}|^pd\mu_1\\
    &\le \dashint_{Q_{2^{j+1}r}(X_0)}|d^\al\D_{x'}u|^pd\mu_1+C(\omega_x(2^{j+1}r))^p,
\end{align*}
where we used \eqref{eq:psi-omega} in the last step. Summing up this estimate over $0\le j\le k-1$ yields
\begin{align*}
    \dashint_{Q_r(X_0)}|d^\al\D_{x'}u|^pd\mu_1\le \dashint_{Q_{2^kr}(X_0)}|d^\al\D_{x'}u|^pd\mu_1+C\sum_{j=1}^k(\omega_x(2^jr))^p.
\end{align*}
By using Hölder's inequality, $d/8<2^kr\le d/4$ and \eqref{eq:psi-omega}, we further have
\begin{align*}
    &\left(\dashint_{Q_r(X_0)}|d^\al\D_{x'}u|^pd\mu_1\right)^{1/p}\lesssim \left(\dashint_{Q_{2^kr}(X_0)}|x_n^\al\D_{x'}u|^pd\mu_1\right)^{1/p}+k^{\frac{1-p}p}\sum_{j=1}^k\omega_x(2^jr)\\
    &\lesssim \left(\dashint_{Q_{2^{k+2}r}(X_0)}|x_n^\al\D_{x'}u|^pd\mu_1\right)^{1/p}+k^{\frac{1-p}p}\sum_{j=1}^k\omega_x(2^jr)\\
    &\lesssim \omega_x(2^{k+2}r)+k^{\frac{1-p}p}\sum_{j=1}^k\omega_x(2^jr)\lesssim k^{\frac{1-p}p}\sum_{j=1}^{k+2}\omega_x(2^jr)\lesssim k^{\frac{1-p}p}\sum_{j=1}^{k+2}(2^{j\be'}\omega_x(r))\\
    &\lesssim k^{\frac{1-p}p}2^{k\be'}\omega_x(r)\lesssim 2^k\omega_x(r)\lesssim \frac{d}r\omega_x(r).
\end{align*}
Combining this with \eqref{eq:d-diff-est} gives \eqref{eq:u-d-est-Diri}, and the proof is complete.
\end{proof}

\subsection{The regularity in time}\label{subsec:reg-time}
In this section, we write for simplicity
$$
d\mu_1=x_n^{-\al^+}dX,\qquad \eta_A=\eta_A^{-\al^+},\qquad \eta_\bg=\eta_\bg^{-\al^+}.
$$
Recall that we used the same notation in Section~\ref{subsec:Schauder-space}, but with $-\al$ in the place of $-\al^+$; see \eqref{eq:not}. We fix a constant $0<\be<1$ and recall $\be'=\frac{\be+1}2$. We introduce several Dini functions that will be used in this section. We simply write $\|x_n^\al\D u\|_{L^\infty(Q_{1/2}^+)}=\|x_n^\al\D u\|_{\infty}$, $\|\uu\|_{L^1(Q_1^+,d\mu_1)}=\|\uu\|_{1,\mu_1}$ and $\|\bg\|_{L^\infty(Q_1^+)}=\|\bg\|_{\infty}$. Set
\begin{align}
    \label{eq:Dini-tilde}\begin{split}
    &\eta(r):=(\|x_n^\al\D u\|_{\infty}+\|\bg\|_{\infty})\eta_A(r)+\eta_\bg(r),\\
    &\tilde\eta(r):=\left(\|\uu\|_{1,\mu_1}+\|\bg\|_{\infty}+\int_0^1\frac{\eta_\bg(\rho)}\rho d\rho\right)(\tilde\eta_A(r)+r^{\be'})+\tilde\eta_\bg(r),\\
    &\hat\eta(r):=\left(\|\uu\|_{1,\mu_1}+\|\bg\|_{\infty}+\int_0^1\frac{\eta_\bg(\rho)}\rho d\rho\right)(\hat\eta_A(r)+r^{\be'})+\hat\eta_\bg(r),\\
    &\bar\eta(r):=\sum_{k=0}^\infty\frac1{2^k}\hat\eta(r/2^k),
\end{split}\end{align}
where $\tilde\eta_\bullet$ and $\hat\eta_\bullet$ are as in \eqref{eq:tilde-eta} and \eqref{eq:hat-eta}, respectively. Note that $\eta\le\tilde\eta\le\hat\eta\le\bar\eta$. Finally, we let
\begin{align}\label{eq:Dini-omega-time}
\omega_t(r):=\bar\eta(r^{1/2}).
\end{align}
Due to Lemma~\ref{lem:appen}, it is easily seen that for $\sigma_0$ as in \eqref{eq:mod-conti-sigma},
$$
\omega_t(r)\lesssim \sigma_0(r^{1/2}).
$$

The main objective in this section is to prove the following regularity result for $\uu$ in time, which will play a crucial role the the higher-order Schauder type estimates in Section~\ref{sec:HO}.

\begin{theorem}\label{thm:reg-time}
    Let $u$, $A$ and $\bg$ be as in Theorem~\ref{thm:reg-Diri}, and $0<\be<1$. Then, for any $0<h<1/4$ and $(t,x)\in Q_{1/2}^+$,
    \begin{align}\label{eq:reg-time}
    \frac{|\delta_{t,h}\uu(t,x)|}{h^{1/2}}\lesssim \omega_t(h).
    \end{align}
\end{theorem}

For $X_0\in \overline{Q_{1/2}^+}$, $0<r<1/8$, and $1\le \g,\de\le n$, we write
$$
\bar g_\g^{X_0,r}(x_n):=\dashint_{Q'_{2r}(X'_0)}g_\g(Y',x_n)dY',\quad (\bar A^{\g\de})^{X_0,r}(x_n):=\dashint_{Q'_{2r}(X'_0)}A^{\g\de}(Y',x_n)dY'.
$$
As before, when there is no confusion, we write for simplicity
$$
d=(x_0)_n.
$$
We define
\begin{align}\label{eq:u_0}
u_0^{X_0,r}(x_n):=\int_d^{x_n}s^{-\al}((\bar A^{nn})^{X_0,r}(s))^{-1}\bar g_n^{X_0,r}(s)\,ds.
\end{align}
When $r>d/4$, we denote
$$
\bA_1^{X_0,r}:=\left\{q(x_n)=\int_0^ds^{-\al}((\bar A^{nn})^{X_0,r}(s))^{-1}\bar g_n^{X_0,r}(s)ds+\int_0^{x_n}s^{-\al}((\bar A^{nn})^{X_0,r}(s))^{-1}\ell \,ds\right\},
$$
where $\ell$ is a constant in $\R^m$. On the other hand, when $r\le d/4$, we let
\begin{align*}
    \mathbb{A}_2^{X_0,r}:=\Bigg\{q(x)&=d^{-\al}\ell_0+d^{-\al}\sum_{\de=1}^{n-1}\ell_\de(x_\de-(x_0)_\de)\\
    &\qquad+\int_{d}^{x_n}\left((\bar A^{nn})^{X_0,r}(s)\right)^{-1}\left(s^{-\al} \ell_n-d^{-\al}\sum_{\de=1}^{n-1}(\bar A^{n\de})^{X_0,r}(s)\ell_\de\right)ds \Bigg\},
\end{align*}
where $\ell_0,\ell_1,\ldots,\ell_n$ are constants in $\R^m$. Now, we fix $0<p<1$ and set
\begin{align}\label{eq:phi-def}
\vp(X_0,r):=\begin{cases}\inf_{q\in \bA_1^{X_0,r}}\left(\dashint_{Q_r^+(X_0)}|x_n^\al(u-u_0^{X_0,r}-q)|^pd\mu_1\right)^{1/p},&r>d/4,\\
\inf_{q\in \bA_2^{X_0,r}}\left(\dashint_{Q_r(X_0)}|x_n^\al(u-u_0^{X_0,r}-q)|^pd\mu_1\right)^{1/p},&r\le d/4.
\end{cases}
\end{align}
Then we consider a function
\begin{align*}
    q^{X_0,r}(x)&=d^{-\al}\ell_0^{X_0,r}+d^{-\al}\sum_{\de=1}^{n-1}\ell_\de^{X_0,r}(x_\de-(x_0)_\de)\\
    &\qquad+\int_d^{x_n}\left((\bar A^{nn})^{X_0,r}(s)\right)^{-1}\left(s^{-\al}\ell_n^{X_0,r}-d^{-\al}\sum_{\de=1}^{n-1}(\bar A^{n\de})^{X_0,r}(s)\ell_\de^{X_0,r}\right)ds,    
\end{align*}
which belongs to $\bA_1^{X_0,r}$ when $r>d/4$ while $\bA_2^{X_0,r}$ when $r\le d/4$ and satisfies
$$
\left(\dashint_{Q_r^+(X_0)}|x_n^\al(u-u_0^{X_0,r}-q^{X_0,r})|^pd\mu_1\right)^{1/p}=\vp(X_0,r).
$$
Here, when $r>d/4$, we take $\ell_\de^{X_0,r}=0$ for $1\le \de\le n-1$ and $\ell_0^{X_0,r}=\int_0^d(d/s)^\al((\bar A^{nn})^{X_0,r}(s))^{-1}(\ell_n^{X_0,r}+\bar g_n^{X_0,r}(s))ds$. When $d=0$, we understand $d^{-\al}\ell_\de^{X_0,r}=0$ for $0\le \de\le n-1$.

As in Section~\ref{subsec:Schauder-space}, we first obtain the estimate of $\vp$ (Lemma~\ref{lem:vp-est-Diri}) by considering the boundary and interior cases. We start with a technical lemma.

\begin{lemma}
    \label{lem:g-A-diff-est}
    Suppose $X_0\in\overline{Q_{1/2}^+}$, $r\in(0,1/2)$ and $\ka\in(0,1)$. Then 
    \begin{align*}
        &\dashint_{Q_{2\ka r}^+(X_0)}x_n^\al\int_d^{x_n}|\bar g_n^{X_0,\ka r}(s)-\bar g_n^{X_0,r}(s)|s^{-\al}dsd\mu_1\lesssim \frac{r\eta_\bg(r)}{\ka^{n+1-\al^-}},\\
        &\dashint_{Q_{2\ka r}^+(X_0)}x_n^\al\int_d^{x_n}|(\bar A^{\g\delta})^{X_0,\ka r}(s)-(\bar A^{\g\delta})^{X_0,r}(s)|s^{-\al}dsd\mu_1\lesssim \frac{r\eta_A(r)}{\ka^{n+1-\al^-}}.
    \end{align*}
\end{lemma}

\begin{proof}
We only prove the first inequality since the second one follows in a similar way. By using the Fubini theorem, we get
\begin{align*}
    &\dashint_{Q_{2\ka r}^+(X_0)}x_n^\al\int_d^{x_n}|\bar g_n^{X_0,\ka r}(s)-\bar g_n^{X_0,r}(s)|s^{-\al}dsd\mu_1\\
    &\le\frac1{\int_{(d-2\ka r)^+}^{d+2\ka r}x_n^{-\al} dx_n}\int_{(d-2\ka r)^+}^{d+2\ka r}\int_{(d-2\ka r)^+}^{x_n}|\bar g_n^{X_0,\ka r}(s)-\bar g_n^{X_0,r}(s)|s^{-\al} dsdx_n\\
    &= \frac1{\int_{(d-2\ka r)^+}^{d+2\ka r}x_n^{-\al} dx_n}\int_{(d-2\ka r)^+}^{d+2\ka r}\int_s^{d+2\ka r}|\bar g_n^{X_0,\ka r}(s)-\bar g_n^{X_0,r}(s)|s^{-\al} dx_nds\\
    &\lesssim \frac{\ka r}{\int_{(d-2\ka r)^+}^{d+2\ka r}x_n^{-\al}dx_n}\int_{(d-2\ka r)^+}^{d+2\ka r}|\bar g_n^{X_0,\ka r}(s)-\bar g_n^{X_0,r}(s)|s^{-\al}ds\\
    &\lesssim \frac1{(\ka r)^n\int_{(d-2\ka r)^+}^{d+2\ka r}x_n^{-\al}dx_n}\int_{Q_{2\ka r}^+(X_0)}|g_n-\bar g_n^{X_0,r}|d\mu_1\\
    &\lesssim \frac{r^{n+1}\int_{(d-2r)^+}^{d+2r}x_n^{-\al}dx_n}{(\ka r)^n\int_{(d-2\ka r)^+}^{d+2\ka r}x_n^{-\al}dx_n}\dashint_{Q_{2r}^+(X_0)}|g_n-\bar g_n^{X_0,r}|d\mu_1.
\end{align*}
Recall $\al=\al^++\al^-$. We claim that $\mathcal{A}:=\frac{\int_{(d-2r)^+}^{d+2r}x_n^{-\al}dx_n}{\int_{(d-2\ka r)^+}^{d+2\ka r}x_n^{-\al}dx_n}\lesssim (1/\ka)^{1-\al^-}$, which combined with the preceding estimate concludes Lemma~\ref{lem:g-A-diff-est}. Indeed, if $d\le\ka r$, then $\mathcal{A}\lesssim\frac{r^{1-\al}}{(\ka r)^{1-\al}}=(1/\ka)^{1-\al}$. On the other hand, if $d\ge r$, then $\mathcal{A}\lesssim\frac{rd^{-\al}}{\ka rd^{-\al}}=1/\ka$. Finally, if $\ka r<d<r$, then $\mathcal{A}\lesssim \frac{r^{1-\al}}{\ka rd^{-\al}}=(r/d)^{-\al}1/\ka\lesssim\begin{cases}
    1/\ka&\text{if }\al\ge0,\\
    (1/\ka)^{1-\al}&\text{if }\al<0.
\end{cases}$
\end{proof}

\begin{lemma}
    \label{lem:vp-bdry-est-Diri}
    Let $0<p<1$, $0<\be<1$, and $\bar X_0\in Q'_{1/2}$. Then there is a constant $\ka\in(0,1/4)$, depending only on $n,m,\al,p,\la$ and $\be$, such that
    \begin{align}
        \label{eq:vp-bdry-est-Diri}
         \vp(\bar X_0,\ka r)\le \ka^{1+\be'}\vp(\bar X_0,r)+Cr\eta(r)
    \end{align}
    for any $0<r<1/4$, where $C=C(n,m,\al,p,\la)>0$.
\end{lemma}

\begin{proof}
Without loss of generality, we may assume $\bar X_0=0$. For simplicity, we write $\bar g_n^{0,r}=\bar g_n^{r}$, $u_0^{0,r}=u_0^{r}$ and $\bA_1^{0,r}=\bA_1^r$, etc., where $u_0^{0.r}$ is as in \eqref{eq:u_0}. We decompose $u=u_0^{r}+v+w$, where $v$ and $w$ are as in the proof of Lemma~\ref{lem:bdry-est-Diri}. We recall that $v$ is a solution of \eqref{eq:v-Diri} and let $V=x_n^\al\bar A^{n\de}D_\de v$ as before. We then define the operator $\hat T^{\rho}$, $0<\rho<1/4$, as follows: given $f\in W^{1,p}(Q_{2\rho}^+;\R^m, x_n^{\al p}d\mu_1)$,
$$
\hat T^{\rho}f(y_n):=\int_0^{y_n}s^{-\al}((\bar A^{nn})^{r}(s))^{-1}(x_n^\al\bar A^{n\de}D_\de f)(0)\,ds.
$$
We write $\hat q_{\rho}:=\hat T^{\rho}v\in \bA_1^r$. We use \cite{DonPha21}*{(4.9) and (5.15)} to deduce
\begin{align*}
    &\|x_n^\al\D_{x'}(v-\hat q_{r})\|_{L^\infty(Q^+_{\ka r})}=\|x_n^\al\D_{x'}v\|_{L^\infty(Q^+_{\ka r})}\le C\ka\left(\dashint_{Q^+_{r/2}}|x_n^\al\D_{x'}v|^pd\mu_1\right)^{1/p},\\
    &\|x_n^\al\sum_{\de=1}^n(\bar A^{n\de})^{r}D_\de(v-\hat q_{r})\|_{L^\infty(Q^+_{\ka r})}=\|V-V(0)\|_{L^\infty(Q^+_{\ka r})}\le \ka r[V]_{C^{1/2,1}(Q^+_{\ka r})}\\
    &\,\,\,\,\qquad\qquad\qquad\qquad\qquad\qquad\qquad\le C\ka\left(\dashint_{Q^+_{r/2}}|x_n^\al\D v|^pd\mu_1\right)^{1/p}.
\end{align*}
Then, it follows by \cite{DonPha21}*{(4.11)},
\begin{align*}
    \|x_n^\al\D(v-\hat q_{r})\|_{L^\infty(Q^+_{\ka r})}\le C\ka\left(\dashint_{Q_{r/2}^+}|x_n^\al \D v|^pd\mu_1\right)^{1/p}\le \frac{C\ka}r\left(\dashint_{Q_r^+}|x_n^\al v|^pd\mu_1\right)^{1/p}.
\end{align*}
Since $v-\hat q_r=0$ on $Q'_r$, by arguing as in \eqref{eq:u-L^infty-bound}, we get
$$
|(v-\hat q_r)(X)|\le C\|x_n^\al\D(v-\hat q_r)\|_{L^\infty(Q^+_{\ka r})}x_n^{1-\al},\quad X\in Q^+_{\ka r}.
$$
Combining the previous two estimates yields that for $X\in Q_{\ka r}^+$,
$$
|x_n^\al(v-\hat q_r)(X)|\le C\|x_n^\al \D(v-\hat q_r)\|_{L^\infty(Q^+_{\ka r})}x_n\le C\ka^2\left(\dashint_{Q_r^+}|x_n^\al v|^pd\mu_1\right)^{1/p}.
$$
For any $q\in\bA_1^r$, as $v-q$ satisfies \eqref{eq:v-Diri} and $\hat T^{r}q=q$, we can replace $v$ with $v-q$ in the above estimate to obtain
\begin{align}
    \label{eq:hom-poly-diff-est-Diri}
    \|x_n^\al(v-\hat q_r)\|_{L^\infty(Q^+_{\ka r})} \le C\ka^2\left(\dashint_{Q_r^+}|x_n^\al(v-q)|^pd\mu_1\right)^{1/p}.
\end{align}
Moreover, it is easily seen that 
\begin{align*}
    \left|((\bar A^{nn})^{\ka r})^{-1}\bar g_n^{\ka r}-((\bar A^{nn})^{r})^{-1}\bar g_n^{r}\right|\lesssim |\bar g_n^{\ka r}-\bar g_n^{r}|+\|g_n\|_\infty|(\bar A^{nn})^{\ka r}-(\bar A^{nn})^{r}|.
\end{align*}
This, along with $\eta=(\|x_n^\al\D u\|_{L^\infty(Q_{1/2}^+)}+\|\bg\|_{L^\infty(Q_1^+)})\eta_A+\eta_\bg$ and Lemma~\ref{lem:g-A-diff-est}, gives
\begin{align}
    \label{eq:u_0-diff-est}
    \left(\dashint_{Q_{\ka r}^+}|x_n^\al(u_0^{\ka r}-u_0^r)|^pd\mu_1\right)^{1/p}\lesssim\frac{r\eta(r)}{\ka^{n+1-\al^-}}.
\end{align}
Next, we recall $\hat q_{r}(x_n)=\int_0^{x_n}s^{-\al}((\bar A^{nn})^{r}(s))^{-1}V(0)\,ds\in \bA_1^r$ and 

\noindent$\hat q_{\ka r}(x_n)=\int_0^{x_n}s^{-\al}((\bar A^{nn})^{\ka r}(s))^{-1}V(0)\,ds\in \bA_1^{\ka r}$.
Then we have by applying Lemma~\ref{lem:g-A-diff-est},
\begin{align*}
    \left(\dashint_{Q_{\ka r}^+}|x_n^\al(\hat q_{\ka r}-\hat q_{r})|^pd\mu_1\right)^{1/p}\le \frac{Cr\eta_A(r)}{\ka^{n+1-\al^-}}|V(0)|.
\end{align*}
To estimate $|V(0)|$, we use the equalities $v=u-u_0^{r}-w$ and $x_n^\al\D u_0^{r}=((\bar A^{nn})^{r})^{-1}\bar g_n^{r}\vec{e}_n$ and apply \cite{DonPha21}*{Proposition~4.2} and \eqref{eq:w-est} to have
\begin{align*}
    |V(0)|&\le \|V\|_{L^\infty(Q_{r/2}^+)}\lesssim \|x_n^\al\D v\|_{L^\infty(Q^+_{r/2})}\lesssim \left(\dashint_{Q_r^+}|x_n^\al\D v|^pd\mu_1\right)^{1/p}\\
    &\lesssim \|x_n^\al\D u\|_\infty+\|x_n^\al\D u_0^{r}\|_\infty+\left(\dashint_{Q_r^+}|x_n^\al\D w|^pd\mu_1\right)^{1/p}\lesssim \|x_n^\al\D u\|_\infty+\|\bg\|_\infty.
\end{align*}
Thus,
\begin{align}
    \label{eq:q-diff-est}
    \left(\dashint_{Q_{\ka r}^+}|x_n^\al(\hat q_{\ka r}-\hat q_r)|^pd\mu_1\right)^{1/p}\lesssim \frac{r\eta(r)}{\ka^{n+1-\al^-}}.
\end{align}
On the other hand, by applying the weak type-(1,1) estimate \eqref{eq:weak-type-2} and using the argument that led to \eqref{eq:w-est}, we get
\begin{align}\label{eq:w-est-2}
\left(\dashint_{Q_r^+}|x_n^\al w|^pd\mu_1\right)^{1/p}\le Cr\left(\|x_n^\al\D u\|_{L^\infty(Q_1^+)}\eta_A(2r)+\eta_\bg(2r)\right).
\end{align}
Now, we denote by $C_\ka$ a constant which may vary from line to line but depends only on $\ka,n,m,\al,p$, and $\la$. From \eqref{eq:hom-poly-diff-est-Diri} - \eqref{eq:w-est-2}, we infer that for any $q\in \bA_1^r$,
\begin{align*}
    &\dashint_{Q_{\ka r}^+}|x_n^\al(u-u_0^{\ka r}-\hat q_{\ka r})|^pd\mu_1\le \dashint_{Q^+_{\ka r}}|x_n^\al(u-u_0^{r}-\hat q_{r})|^pd\mu_1+C_\ka(r\eta(r))^p\\
    &\le \dashint_{Q_{\ka r}^+}|x_n^\al(v-\hat q_{r})|^pd\mu_1+\dashint_{Q_{\ka r}^+}|x_n^\al w|^pd\mu_1+C_\ka(r\eta(r))^p\\
    &\le C\ka^{2p}\dashint_{Q_{r}^+}|x_n^\al(v-q)|^pd\mu_1+C_\ka\dashint_{Q_r^+}|x_n^\al w|^pd\mu_1+C_\ka(r\eta(r))^p\\
    &\le C\ka^{2p}\dashint_{Q_r^+}|x_n^\al(u-u_0^{r}-q)|^pd\mu_1+C_\ka(r\eta(r))^p.
\end{align*}
As before, this implies \eqref{eq:vp-bdry-est-Diri}.
\end{proof}

\begin{lemma}
    \label{lem:vp-int-est-Diri}
    Let $p,\be,\be'$, and $\ka$ be as in Lemma~\ref{lem:vp-bdry-est-Diri}. Then, for any $X_0\in Q^+_{1/2}$ and $0<r<d/4$,
    \begin{align}\label{eq:vp-int-est-Diri}
    \vp(X_0,\ka r)\le \ka^{1+\be'}\vp(X_0,r)+Cr\eta(r).
    \end{align}
\end{lemma}

\begin{proof}
To obtain Lemma~\ref{lem:vp-int-est-Diri}, we will apply \cite{DonJeo24}*{Lemma~3.10}. To this end, we observe that
\begin{align*}
    D_\g(x_n^\al A^{\g\de}D_\de u)-x_n^\al\partial_tu=\div(x_n^\al\tilde\bg) \quad\text{in }Q_r(X_0),\quad\text{where }\tilde\bg:=x_n^{-\al}\bg.
\end{align*}
Since $3/4<x_n/d<5/4$ in $Q_r(X_0)$, we have
\begin{align}\label{eq:tilde-g-est}
    \dashint_{Q_r(X_0)}\left|\tilde\bg(X)-\dashint_{Q'_r(X'_0)}\tilde\bg(Y',x_n)dY'\right|d\mu_1(X)\lesssim d^{-\al}\eta_\bg(r).
\end{align}
Note that $u_0^{X_0,r}$ and $\mathbb{A}_2^{X_0,r}$ can be rewritten as
\begin{align*}
    &u_0^{X_0,r}(x_n)=\int_d^{x_n}((\bar A^{nn})^{X_0,r}(s))^{-1}\dashint_{Q'_{2r}(X'_0)}\tilde g_n(Y',s)dY'ds,\\
    &\mathbb{A}_2^{X_0,r}=\bigg\{q\,:\,q(x)=\ell_0+\sum_{\de=1}^{n-1}\ell_\de(x_\de-(x_0)_\de)\\
    &\qquad\qquad\qquad\qquad\qquad+\int_{d}^{x_n}\left((\bar A^{nn})^{X_0,r}(s)\right)^{-1}\left((d/s)^\al \ell_n-\sum_{\de=1}^{n-1}(\bar A^{n\de})^{X_0,r}(s)\ell_\de\right)ds \bigg\}.
\end{align*}
We set
$$
\tilde\vp(X_0,r):=\inf_{q\in \mathbb{A}_{2}^{X_0,r}}\left(\dashint_{Q_r(X_0)}|u- u_0^{X_0,r}-q|^pd\mu_1\right)^{1/p}.
$$
By using \eqref{eq:tilde-g-est} and following the argument leading to \cite{DonJeo24}*{Lemma~3.10}, we obtain
\begin{align*}\begin{split}
    \tilde\vp(X_0,\ka r)&\lesssim \ka^{1+\be'}\tilde\vp(X_0,r)+rd^{-\al}\tilde\eta_\bg(r)\\
    &\qquad+\left(\|\D u\|_{L^\infty(Q_{2r}(X_0))}+\|\tilde\bg\|_{L^\infty(Q_{2r}(X_0))}\right)r\tilde\eta_A(r).
\end{split}\end{align*}
By multiplying $d^\al$ in both sides and using $1/2<x_n/d<3/2$ in $Q_{2r}(X_0)$, we obtain \eqref{eq:vp-int-est-Diri}.    
\end{proof}

\begin{lemma}
    \label{lem:l-i-bound}
Let $\sigma:(0,1)\to[0,\infty)$ be a Dini function satisfying \eqref{eq:alm-mon} and $\sigma\ge\eta$, where $\eta$ is as in \eqref{eq:Dini-tilde}. Assume that for some point $X_0\in \overline{Q_{1/2}^+}$,
\begin{align}\label{eq:vp-bound-assump}
\vp(X_0,r)\le r\sigma(r),\quad 0<r<1/8.
\end{align}
Then there is a constant $C>0$, depending only on $n,m,\la,\al,p,$ such that for every $0<r<1/16$,
\begin{align}
    \label{eq:u-poly-grad-est}
    &\left(\dashint_{Q_r^+(X_0)}|x_n^\al\D(u-u_0^{X_0,r}-q^{X_0,r})|^pd\mu_1\right)^{1/p}\le C\sigma(r),\\
    \label{eq:U-l-diff-est}
    & \left(\dashint_{Q_r^+(X_0)}|U-\ell_n^{X_0,r}|^pd\mu_1\right)^{1/p}\le C\sigma(r),
\end{align}
and  for every $1\le\delta\le n$,
\begin{align}\label{eq:l_i-est}
    |\ell_\de^{X_0,r}-\ell_\de^{X_0,2r}|\le C\sigma(r)\quad\text{and}\quad|\ell_\de^{X_0,r}|\le C\left(\|\uu\|_\infty+\|\bg\|_\infty+\int_0^1\frac{\sigma(\rho)}\rho d\rho\right).
\end{align}
\end{lemma}

\begin{proof}
It is easily seen that $u^{X_0,r}:=u-u_0^{X_0,r}-q^{X_0,r}$ satisfies in $Q^+_{2r}(X_0)$
$$
D_\g(x_n^\al(\bar A^{\g\de})^{X_0,r}D_\de u^{X_0,r})-x_n^\al\partial_tu^{X_0,r}=D_\g(x_n^\al((\bar A^{\g\de})^{X_0,r}-A^{\g\de})D_\de u+g_\g-\bar g_\g^{X_0,r}))
$$
with $u^{X_0,r}=0$ on $Q_{2r}'(X_0)$ if $Q_{2r}(X_0)\cap Q'_4$ is nonempty. We decompose $u^{X_0,r}=w^{X_0,r}+v^{X_0,r}$ in a similar way as in the proof of Lemma~\ref{lem:bdry-est-Diri}. Then, by applying the weak type-(1,1) estimates in Lemma~\ref{lem:weak-type-(1,1)-Diri}, we obtain
\begin{align}
    \label{eq:w-X_0-est-1}\left(\dashint_{Q^+_{r}(X_0)}| x_n^\al w^{X_0,r}|^pd\mu_1\right)^{1/p}\lesssim r\eta(r),\quad \left(\dashint_{Q^+_{r}(X_0)}|x_n^\al\D w^{X_0,r}|^pd\mu_1\right)^{1/p}\lesssim \eta(r).
\end{align}
Regarding $v$, we have
\begin{align*}
    \left(\dashint_{Q^+_{r}(X_0)}|x_n^\al\D v^{X_0,r}|^pd\mu_1\right)^{1/p}&\lesssim \frac{1}{r}\left(\dashint_{Q^+_{\frac32r}(X_0)}|x_n^\al v^{X_0,r}|^pd\mu_1\right)^{1/p}\lesssim \sigma(r),
\end{align*}
where we applied \cite{DonPha21}*{(4.11)} with iteration in the first inequality and \eqref{eq:vp-bound-assump} and \eqref{eq:w-X_0-est-1} with the identity $v^{X_0,r}=u^{X_0,r}-w^{X_0,r}$ in the second inequality. Combining this with the second estimate in \eqref{eq:w-X_0-est-1} yields \eqref{eq:u-poly-grad-est}.

Next, we prove \eqref{eq:l_i-est} for $\de=1,\ldots,n-1$. To this end, we apply Lemma~\ref{lem:g-A-diff-est} to get
$$
\left(\dashint_{Q_r^+(X_0)}|x_n^\al(u_0^{X_0,r}-u_0^{X_0,2r})|^pd\mu_1\right)^{1/p}\lesssim r\eta(r).
$$
By using this estimate, together with \eqref{eq:vp-bound-assump} and $\sigma\ge\eta$, we obtain
\begin{align}\label{eq:q-l-est-1}\begin{split}
    &\left(\dashint_{Q_r^+(X_0)}|x_n^\al(q^{X_0,r}-q^{X_0,2r})|^pd\mu_1\right)^{1/p}\\
    &\lesssim \left(\dashint_{Q_r^+(X_0)}|x_n^\al(u-u_0^{X_0,r}-q^{X_0,r})|^pd\mu_1\right)^{1/p}+\left(\dashint_{Q_r^+(X_0)}|x_n^\al(u-u_0^{X_0,2r}-q^{X_0,2r})|^pd\mu_1\right)^{1/p}\\
    &\qquad+\left(\dashint_{Q_r^+(X_0)}|x_n^\al(u_0^{X_0,r}-u_0^{X_0,2r})|^pd\mu_1\right)^{1/p}\\
    &\lesssim \vp(X_0,r)+\vp(X_0,2r)+r\eta(r)\lesssim r\sigma(r).
\end{split}\end{align}
As $\ell_\de^{X_0,r}=\ell_\de^{X_0,2r}=0$ for any $1\le\de\le n-1$ when $r>d/4$, we may assume $r\le d/4$. We may also assume without loss of generality $X_0'=0$. We then observe
\begin{multline*}
    2(x_n/d)^\al(\ell_1^{X_0,r}-\ell_1^{X_0,2r})x_1\\
    =x_n^\al(q^{X_0,r}-q^{X_0,2r})(x)-x_n^\al(q^{X_0,r}-q^{X_0,2r})(-x_1,x_2,\ldots,x_n),
\end{multline*}
which combined with \eqref{eq:q-l-est-1} gives
\begin{align*}
    r|\ell_1^{X_0,r}-\ell_1^{X_0,2r}|\lesssim\left(\dashint_{Q_r^+(X_0)}|x_n^\al(q^{X_0,r}-q^{X_0,2r})|^pd\mu_1\right)^{1/p}\lesssim r\sigma(r).
\end{align*}
By repeating the above process for $2\le \de\le n-1$, we get the first inequality in \eqref{eq:l_i-est}. For the second one, we take $k\in\mathbb{N}$ such that $1/8< 2^kr\le1/4$. Since $d/4\le 1/8<2^kr$, we have $\ell_\de^{X_0,2^kr}=0$ for $1\le\de\le n-1$. We combine this and the first inequality in \eqref{eq:l_i-est} to conclude
\begin{align}\label{eq:l-bound-sum}
    |\ell_\delta^{X_0,r}|\le \sum_{i=0}^{k-1}|\ell_\delta^{X_0,2^ir}-\ell_\delta^{X_0,2^{i+1}r}|+|\ell_\delta^{X_0,2^kr}|\lesssim \int_0^1\frac{\sigma(\rho)}\rho d\rho.
\end{align}

It remains to prove \eqref{eq:U-l-diff-est} and \eqref{eq:l_i-est} for $\de=n$. From a direct computation, we have
\begin{align*}
&x_n^\al \sum_{\de=1}^n(\bar A^{n\de})^{X_0,r}D_\de(u-u_0^{X_0,r}-q^{X_0,r})\\
&=(U-\ell_n^{X_0,r})+(g_n-\bar g_n^{X_0,r})+x_n^\al((\bar A^{n\de})^{X_0,r}-A^{n\de})D_\de u.
\end{align*}
Combining this with \eqref{eq:u-poly-grad-est} implies \eqref{eq:U-l-diff-est}. This in turn gives 
\begin{align*}
    |\ell_n^{X_0,r}-\ell_n^{X_0,2r}|\lesssim\left(\dashint_{Q_r^+(X_0)}|U-\ell_n^{X_0,r}|^p+|U-\ell_n^{X_0,2r}|^pd\mu_1 \right)^{1/p}\lesssim \sigma(r).
\end{align*}
As above, let $k\in\mathbb{N}$ be such that $1/8<2^kr\le1/4$. Recall that $q^{X_0,2^kr}(x_n)=\int_0^ds^{-\al}\left((\bar A^{nn})^{X_0,2^kr}(s)\right)^{-1}\bar g_n^{X_0,2^kr}(s)\,ds+\int_0^{x_n}s^{-\al}\left((\bar A^{nn})^{X_0,2^kr}(s)\right)^{-1}\ell_n^{X_0,2^kr}ds$, and note that since $\al<1$, 
$$
\left(\dashint_{Q_{2^kr}^+(X_0)}\left|x_n^\al\int_0^ds^{-\al}\left((\bar A^{nn})^{X_0,r}(s)\right)^{-1}\bar g_n^{X_0,r}(s)
ds\right|^pd\mu_1\right)^{1/p}\lesssim \|\bg\|_\infty.
$$
Then we have
\begin{align*}
    |\ell_n^{X_0,2^kr}|&\lesssim \left(\dashint_{Q_{2^kr}^+(X_0)}|x_n^\al q^{X_0,2^kr}(x_n)|^pd\mu_1\right)^{1/p}+\|\bg\|_\infty\\
    &\lesssim \vp(X_0,2^kr)+\left(\dashint_{Q_{2^kr}^+(X_0)}|x_n^\al(u-u_0^{X_0,2^kr})|^pd\mu_1\right)^{1/p}+\|\bg\|_\infty\\
    &\lesssim \left(\dashint_{Q_{2^kr}^+(X_0)}|x_n^\al(u-u_0^{X_0,2^kr})|^pd\mu_1\right)^{1/p}+\|\bg\|_\infty\\
    &\lesssim \|\uu\|_\infty+\|\bg\|_\infty.
\end{align*}
Therefore, we conclude
\begin{align*}
    |\ell_n^{X_0,r}|\le \sum_{i=0}^{k-1}|\ell_n^{X_0,2^ir}-\ell_n^{X_0,2^{i+1}r}|+|\ell_n^{X_0,2^kr}|\lesssim\int_0^1\frac{\sigma(\rho)}\rho d\rho+\|\uu\|_\infty+\|\bg\|_\infty.
\end{align*}
This completes the proof.
\end{proof}

%%%%%%%%%%%%%%%%%%%%%%%%%%%%%%%%%%%%%%%%%%%%%%%%%%%%%%%%%%%%%%

\begin{lemma}
    \label{lem:vp-est-Diri}
    If $X_0\in Q_{1/2}^+$ and $0<r<\ka/20$, then
    \begin{align}
        \label{eq:vp-est-Diri}
        \vp(X_0,r)\le Cr\hat\eta(r).
    \end{align}
\end{lemma}

\begin{proof}
We assume without loss of generality $X'_0=0$, and write for simplicity $q^p=q^{0,p}$, $\bar g_n^\rho=\bar g_n^{0,\rho}$ and $(\bar A^{n\de})^\rho=(\bar A^{n\de})^{0,\rho}$. 

We first claim that for any $0<R<1/4$,
\begin{align}
    \label{eq:vp-bdry-est-2}
    \vp(0,R)\lesssim R\tilde\eta(R).
\end{align}
Indeed, we choose a nonnegative integer $j\ge0$ such that $\ka^{-j}R< 1/4\le \ka^{-(j+1)}R$. By applying \eqref{eq:vp-bdry-est-Diri}, we get
\begin{align*}
    \vp(0,R)&\le \ka^{(1+\be')j}\vp(0,\ka^{-j}R)+C\sum_{i=1}^j\ka^{(1+\be')(i-1)}\ka^{-i}R\eta(\ka^{-i}R)\\
    &\lesssim R^{1+\be'}\left(\dashint_{Q^+_{\ka^{-j}R}}|x_n^\al(u-u_0^{X_0,r})|^pd\mu_1\right)^{1/p}+R\sum_{i=1}^j\ka^{\be'i}\eta(\ka^{-i}R)\\
    &\lesssim (\|\uu\|_{1,\mu_1}+\|\bg\|_\infty)R^{1+\be'}+R\tilde\eta(R)\\
    &\lesssim R\tilde\eta(R).
\end{align*}

Next, we consider two cases either $r\ge \ka d/4$ or $0<r<\ka d/4$.

\medskip\noindent\emph{Case 1}. Suppose $r\ge\ka d/4$. For $R:=(1+4/\ka)r$, we have by \eqref{eq:vp-bdry-est-2},
\begin{align}\label{eq:vp-bdry-est-3}
\vp(0,R)\lesssim R\tilde\eta(R)\lesssim r\tilde\eta(r).
\end{align}
Then, by Lemmas~\ref{lem:gradient-unif-est-Diri} and \ref{lem:l-i-bound}, 
\begin{align}
    \label{eq:l-est}
    |\ell_\de^R|\lesssim \|\uu\|_{1,\mu_1}+\|\bg\|_\infty+\int_0^1\frac{\hat\eta(\rho)}\rho d\rho,\quad 1\le\de\le n.
\end{align}
Moreover, from $R=r+\frac{4r}\ka\ge r+d$, we see $Q_r^+(X_0)\subset Q_R^+$. Then, we use the minimality of $\vp$ and the equalities $\bar g_n^{X_0,r}=\bar g_n^r$ and $(\bar A^{nn})^{X_0,r}=(\bar A^{nn})^r$ to get
\begin{align*}
    &\vp(X_0,r)\\
    &\le\left(\dashint_{Q_r^+(X_0)}\left|x_n^\al\left(u-\int_0^{x_n}s^{-\al}\left((\bar A^{nn})^r(s)\right)^{-1}(\bar g_n^r(s)+\ell_n^R)ds\right)\right|^pd\mu_1\right)^{1/p}\\
    &\lesssim \left(\dashint_{Q_r^+(X_0)}\left|x_n^\al\left(u-\int_0^{x_n}s^{-\al}\left((\bar A^{nn})^R(s)\right)^{-1}(\bar g_n^R(s)+\ell_n^R)ds\right)\right|^pd\mu_1\right)^{1/p}\\
    &\qquad+\dashint_{Q_r^+(X_0)}\left|x_n^\al\int_0^{x_n}s^{-\al}\left((\bar A^{nn})^R(s)\right)^{-1}(\bar g_n^R(s)-\bar g_n^r(s))ds\right|d\mu_1\\
    &\qquad+\dashint_{Q_r^+(X_0)}\left|x_n^\al\int_0^{x_n}s^{-\al}\left(\left((\bar A^{nn})^r(s)\right)^{-1}-\left((\bar A^{nn})^R(s)\right)^{-1}\right)(\bar g_n^r(s)+\ell_n^R)ds\right|d\mu_1\\
    &\lesssim \vp(0,R)+r\hat\eta(r)\lesssim r\hat\eta(r),
\end{align*}
where we used Lemma~\ref{lem:g-A-diff-est}, \eqref{eq:vp-bdry-est-3}, \eqref{eq:l-est}, and $Q_r^+(X_0)\subset Q_R^+$.

\medskip\noindent\emph{Case 2}. If $0<r<\ka d/4$, then we can argue as in Case 2 in the proof of \cite{DonJeo24}*{Lemma~3.12} to obtain \eqref{eq:vp-est-Diri}.
\end{proof}

By combining the previous two lemmas, we obtain the following result, which provides crucial estimates for the proof of the main result, Theorem~\ref{thm:reg-time}, in this section. 

\begin{lemma}
    \label{lem:diff-est}
    For $X_0\in Q_{1/2}^+$ and $0<r<1/12$,
    \begin{align}
        \label{eq:diff-est}\begin{split}
        &|\ell_\de^{X_0,r}-\ell_\de^{X_0,2r}|\lesssim \hat\eta(r),\quad 1\le\de\le n,\\
        &|\ell_\de^{X_0,r}|\lesssim \|\uu\|_\infty+\|\bg\|_\infty+\int_0^1\frac{\eta(\rho)}\rho d\rho,\quad 1\le\de\le n,\\
        &\left(\dashint_{Q_r^+(X_0)}|x_n^\al D_\de u-(x_n/d)^\al\ell_\de^{X_0,r}|d\mu_1\right)^{1/p}\lesssim \hat\eta(r),\quad 1\le\de\le n-1,\\
        &\left(\dashint_{Q_r^+(X_0)}|U-\ell_n^{X_0,r}|^pd\mu_1\right)^{1/p}\lesssim\hat\eta(r),\\
        &|\ell_0^{X_0,r}-\ell_0^{X_0,2r}|\lesssim r\hat\eta(r),\\
        &|\uu(X_0)-\ell_0^{X_0,r}|\lesssim r\bar\eta(r).
    \end{split}\end{align}
\end{lemma}

\begin{proof}
The first four estimates in \eqref{eq:diff-est} follow from Lemma~\ref{lem:l-i-bound}, which is applicable thanks to Lemma~\ref{lem:vp-est-Diri}, and Lemma~\ref{lem:appen}. To prove the last two estimates, we assume without loss of generality $X'_0=0$, and consider two cases either $r\le \frac34d$ or $r>\frac34d$. 

\medskip\noindent\emph{Case 1.} If $r\le \frac34d$, we observe that
\begin{align*}
    &(x_n/d)^\al(\ell_0^{X_0,r}-\ell_0^{X_0,2r})\\
    &=x_n^\al(q^{X_0,r}-q^{X_0,2r})-(x_n/d)^\al\sum_{\de=1}^{n-1}(\ell_\de^{X_0,r}-\ell_\de^{X_0,2r})x_\de\\
    &\qquad-x_n^\al \int_d^{x_n}\Bigg[\left((\bar A^{nn})^{X_0,r}(s)\right)^{-1}\left(s^{-\al}\ell_n^{X_0,r}-d^{-\al}\sum_{\de=1}^{n-1}(\bar A^{n\de})^{X_0,r}(s)\ell_\de^{X_0,r}\right)\\
    &\qquad\qquad\qquad-\left((\bar A^{nn})^{X_0,2r}(s)\right)^{-1}\left(s^{-\al}\ell_n^{X_0,2r}-d^{-\al}\sum_{\de=1}^{n-1}(\bar A^{n\de})^{X_0,2r}(s)\ell_\de^{X_0,2r}\right)\Bigg]ds.
\end{align*}
We then use \eqref{eq:q-l-est-1}, Lemma~\ref{lem:g-A-diff-est}, and the first two estimates in \eqref{eq:diff-est}, along with $1/4\le x_n/d\le 7/4$, to obtain $|\ell_0^{X_0,r}-\ell_0^{X_0,2r}|\lesssim r\hat\eta(r)$.

\medskip\noindent\emph{Case 2.} Suppose $r>\frac34d$. In this case, we have
\begin{align*}
    &\ell_0^{X_0,r}=\int_0^d(d/s)^\al\left((\bar A^{nn})^{X_0,r}(s)\right)^{-1}(\ell_n^{X_0,r}+\bar g_n^{X_0,r}(s))ds,\\
    &\ell_0^{X_0,2r}=\int_0^d(d/s)^\al\left((\bar A^{nn})^{X_0,2r}(s)\right)^{-1}(\ell_n^{X_0,2r}+\bar g_n^{X_0,2r}(s))ds.
\end{align*}
Thus
\begin{align*}
    &\ell_0^{X_0,r}-\ell_0^{X_0,2r}\\
    &=\int_0^d(d/s)^\al\left((\bar A^{nn})^{X_0,r}(s)\right)^{-1}(\ell_n^{X_0,r}-\ell_n^{X_0,2r})ds\\
    &+\int_0^d(d/s)^\al\left((A^{nn})^{X_0,r}(s)\right)^{-1}(\bar g_n^{X_0,r}(s)-\bar g_n^{X_0,2r}(s))ds\\
    &+\int_0^d(d/s)^\al\left[\left((\bar A^{nn})^{X_0,r}(s)\right)^{-1}-\left((\bar A^{nn})^{X_0,2r}(s)\right)^{-1}\right](\ell_n^{X_0,2r}+\bar g_n^{X_0,2r}(s))ds\\
    &=:I+II+III.
\end{align*}
We have 
$$
|I|\lesssim \hat\eta(r)\int_0^d(d/s)^\al ds\lesssim d\hat\eta(r)\lesssim r\hat\eta(r).
$$
To bound $|II|$, we consider two cases either $\al\ge0$ or $\al<0$. 

If $\al\ge0$, then for any $x_n\in (d,1/2)$, we have $(d/x_n)^\al\le1$, thus
\begin{align*}
    |II|&\lesssim \int_0^{x_n}(x_n/s)^\al(d/x_n)^\al|\bar g_n^{X_0,r}(s)-\bar g_n^{X_0,2r}(s)|ds\\
    &\lesssim \int_0^{x_n}(x_n/s)^\al|\bar g_n^{X_0,r}(s)-\bar g_n^{X_0,2r}(s)|ds.
\end{align*}
Therefore, 
$$|II|\lesssim \dashint_{Q_{r/3}(0,d+r/3)}x_n^\al\int_0^{x_n}|\bar g_n^{X_0,r}(s)-\bar g_n^{X_0,2r}(s)|s^{-\al}dsd\mu_1(X).
$$
From $r>\frac34d$, we have $Q_{r/3}(0,d+r/3)\subset Q^+_{d+\frac23r}\subset Q^+_{2r}$. Thus, we have by Lemma~\ref{lem:g-A-diff-est},
$$
|II|\lesssim\dashint_{Q_{2r}^+}x_n^\al\int_0^{x_n}|\bar g_n^{X_0,r}(s)-\bar g_n^{X_0,2r}(s)|s^{-\al}dsd\mu_1\lesssim r\eta(r).
$$

Now, suppose $\al<0$. Then $(d/s)^\al<1$ for $s\in(0,d)$. Thus we have
$$
|II|\lesssim\int_0^d|\bar g_n^{X_0,r}(s)-\bar g_n^{X_0,2r}(s)|ds.
$$
We then argue as in the case when $\al\ge0$ and use Lemma~\ref{lem:partial-DMO-weight-rela} to deduce 
$$
|II|\lesssim\dashint_{Q_{2r}^+}\int_0^{x_n}|\bar g_n^{X_0,r}(s)-\bar g_n^{X_0,2r}(s)|dsdX\lesssim r\eta(r).
$$
Similarly, we can show that $|III|\lesssim r\hat\eta(r)$.

Finally, note that $(x_n^\al(u-u_0^{X_0,r}-q^{X_0,r}))(X_0)=(\uu-(x_n/d)^\al\ell_0^{X_0,r})(X_0)$, which combined with Lemma~\ref{lem:vp-est-Diri} implies that $\ell_0^{X_0,r}\to \uu(X_0)$ as $r\to0$. This and the fifth estimate in \eqref{eq:diff-est} yield the last estimate in \eqref{eq:diff-est}.    
\end{proof}

Now we are ready to prove Theorem~\ref{thm:reg-time}.

\begin{proof}[Proof of Theorem~\ref{thm:reg-time}]
It suffices to show that for any $(t_0,x_0)\in Q_{1/2}^+$ and $0<r<\ka/32$,
$$
|\uu(t_0,x_0)-\uu(t_0-r^2,x_0)|\lesssim r\bar\eta(r),
$$
where $\bar\eta$ is as in \eqref{eq:Dini-tilde}. To prove it, we write $X_0=(t_0,x_0)$ and $X_1=(t_0-r^2,x_0)$ and assume without loss of generality $x_0'=0$. From \eqref{eq:diff-est}, we have
\begin{align*}
    |\uu(X_0)-\uu(X_1)|&\le |\uu(X_0)-\ell_0^{X_0,2r}|+|\uu(X_1)-\ell_0^{X_1,2r}|+|\ell_0^{X_0,2r}-\ell_0^{X_1,2r}|\\
    &\le Cr\bar\eta(r)+|\ell_0^{X_0,2r}-\ell_0^{X_1,2r}|.
\end{align*}
Then, it is sufficient to show that
\begin{align}
    \label{eq:l_0-diff-center-est}
    |\ell_0^{X_0,2r}-\ell_0^{X_1,2r}|\le Cr\bar\eta(r).
\end{align}
To this end, we claim that for any $0<s<1/4$ and $1\le\de\le n$,
\begin{align}
    \label{eq:diff-center-est}
    \begin{split}
        &\dashint_{Q_r^+(X_1)}x_n^\al\int_d^{x_n}|\bar g_n^{X_0,2r}(s)-\bar g_n^{X_1,2r}(s)|s^{-\al}dsd\mu_1\lesssim r\bar\eta(r),\\
        &\dashint_{Q_r^+(X_1)}x_n^\al\int_d^{x_n}|(\bar A^{n\de})^{X_0,2r}(s)-(\bar A^{n\de})^{X_1,2r}(s)|s^{-\al}dsd\mu_1\lesssim r\bar\eta(r),\\
        &|\ell_\de^{X_0,2r}|+|\ell_\de^{X_1,2r}|\lesssim \|x_n^\al\D u\|_{\infty}+\|\bg\|_\infty+\int_0^1\frac{\eta(\rho)}\rho d\rho,\\
        &|\ell_\de^{X_0,2r}-\ell_\de^{X_1,2r}|\lesssim\bar\eta(r).
    \end{split}
\end{align}
Indeed, by applying the triangle inequality $|\bar g_n^{X_0,2r}-\bar g_n^{X_1,2r}|\le |\bar g_n^{X_0,2r}-\bar g_n^{X_0,4r}|+|\bar g_n^{X_0,4r}-\bar g_n^{X_1,2r}|$ and the inclusion $Q^+_{2r}(X_1)\subset Q^+_{4r}(X_0)$ and arguing as in the proof of Lemma~\ref{lem:g-A-diff-est}, we can obtain the first estimate in \eqref{eq:diff-center-est}. The second estimate follows in a similar way. The third one in \eqref{eq:diff-center-est} can be found in \eqref{eq:diff-est}. Regarding the last estimate, when $\de=n$, we use \eqref{eq:diff-est} to have
\begin{align*}
    |\ell_n^{X_0,2r}-\ell_n^{X_1,2r}|&\le \left(\dashint_{Q_r^+(X_1)}|\ell_n^{X_0,2r}-U|^p+|U-\ell_n^{X_1,2r}|^pd\mu_1\right)^{1/p}\\
    &\lesssim \left(\dashint_{Q_{2r}^+(X_0)}|\ell_n^{X_0,2r}-U|^pd\mu_1\right)^{1/p}+\left(\dashint_{Q_{2r}^+(X_1)}|U-\ell_n^{X_1,2r}|^pd\mu_1\right)^{1/p}\\
    &\lesssim \hat\eta(r).
\end{align*}
On the other hand, when $1\le\de\le n-1$, we use the triangle inequality $|x_n^\al(q^{X_0,2r}-q^{X_1,2r})|\le |x_n^\al(u-u_0^{X_0,2r}-q^{X_0,2r})|+|x_n^\al(u-u_0^{X_1,2r}-q^{X_1,2r})|+|x_n^\al(u_0^{X_0,2r}-u_0^{X_1,2r})|$ and apply \eqref{eq:vp-est-Diri} and the first two estimates in \eqref{eq:diff-center-est} to deduce
\begin{align}
    \label{eq:q-diff-center-est}
    \left(\dashint_{Q_r^+(X_1)}|x_n^\al(q^{X_0,2r}-q^{X_1,2r})|^pd\mu_1\right)^{1/p}\lesssim r\bar\eta(r).
\end{align}
By using this estimate and the fact that $X_0$ and $X_1$ have the same spatial components, i.e., $x_0=x_1$, we can follow the argument obtaining \eqref{eq:l_i-est} to obtain $|\ell_\de^{X_0,2r}-\ell_\de^{X_1,2r}|\lesssim\bar\eta(r)$, $1\le\de\le n-1$.

Now, by using \eqref{eq:diff-center-est}, \eqref{eq:q-diff-center-est}, and $x_0=x_1$, we can follow the argument deriving the fifth estimate in \eqref{eq:diff-est} to conclude \eqref{eq:l_0-diff-center-est}.    
\end{proof}

Before closing this section, we establish the continuity of $\frac{\delta_{t,h}\uu}{h^{1/2}}$ by using Theorem~\ref{thm:reg-time}.

\begin{theorem}
    \label{thm:diff-quo-conti}
    Let $u$, $A$, and $\bg$ be as in Theorem~\ref{thm:reg-Diri}, and $0<\be<1$. Then $\frac{\delta_{t,h}\uu}{h^{1/2}}\in C_X(\overline{Q_{1/2}^+})$ uniformly in $h\in(0,1/4)$.
\end{theorem}

\begin{proof}
We fix $0<\rho<1$, $1\le\g\le n$, and $(t,x)\in \overline{Q_{1/2}^+}$. We then let
$$
u^*_{\g,\rho}:=\frac{\delta_{t,h}\uu(t,x)-\delta_{t,h}\uu(t,x-\rho e_\g)}{h^{1/2}}.
$$
Suppose $h<\rho^2$. Since $\frac{|\delta_{t,h}\uu(t,x)|}{h^{1/2}}\lesssim \omega_t(h)$ by Theorem~\ref{thm:reg-time}, we have by Lemma~\ref{lem:appen},
\begin{align}
    \label{eq:U-sigma-est}
    |u^*_{\g,\rho}|\lesssim\omega_t(h)=\sum_{k=0}^\infty\frac1{2^k}\hat\eta(h^{1/2}/{2^k})\lesssim\sum_{k=0}^\infty\frac1{2^k}\int_0^{\frac{\rho}{2^k}}\frac{\hat\eta(s)}sds\lesssim \int_0^\rho\frac{\hat\eta(s)}sds\le \sigma_0(\rho).
\end{align}
On the other hand, when $h\ge\rho^2$, we have that for $\omega_x$ as in \eqref{eq:omega_x},
\begin{align*}
    |u^*_{\g,\rho}|&=\frac{|(\uu(t,x)-\uu(t,x-\rho e_\g))-(\uu(t-h,x)-\uu(t-h,x-\rho e_\g))|}{h^{1/2}}\\
    &\le \frac{\rho}{h^{1/2}}\int_0^1|D_\g \uu(t,x-\tau\rho e_\g)-D_\g \uu(t-h,x-\tau\rho e_\g)|d\tau\\
    &\lesssim\frac{\rho\,\omega_x(h^{1/2})}{h^{1/2}}.
\end{align*}
Here, in the last inequality, we used \eqref{eq:u-reg-Diri} when $1\le\g\le n-1$. When $\g=n$, we used
\begin{align}
    \label{eq:u_alpha}\begin{split}
    D_n\uu&=x_n^\al D_nu+\al x_n^{\al-1}u=x_n^\al D_nu+\al x_n^\al\int_0^1D_nu(t,x',x_n\tau)\,d\tau\\
    &=x_n^\al D_nu+\al\int_0^1(x_n^\al D_nu)(t,x',x_n\tau)\tau^{-\al}\,d\tau,
\end{split}\end{align}
and that $x_n^\al D_nu$ has a modulus of continuity bounded by $C\omega_x$ with respect to $t$ by \eqref{eq:u-reg-Diri}.

By using the above estimate and the monotonicity of $r\longmapsto \frac{\omega_x(r)}r$, which follows from that of  $r\longmapsto\frac{\hat\eta_\bullet(r)}{r^{\be'}}$ and $\be'<1$, we further have
$$
|u^*_{\g,\rho}|\lesssim \frac{\rho\omega_x(h^{1/2})}{h^{1/2}}\lesssim \omega_x(\rho).
$$
Therefore, $\frac{\delta_{t,h}\uu}{h^{1/2}}\in C_x(\overline{Q_{1/2}^+})$ uniformly in $h$.

Next, to prove $\frac{\delta_{t,h}\uu}{h^{1/2}}\in C_t(\overline{Q_{1/2}^+})$ uniformly in $h$, we let
$$
u^*_{t,\rho}:=\frac{\delta_{t,h}\uu(t,x)-\delta_{t,h}\uu(t-\rho,x)}{h^{1/2}}.
$$
If $h<\rho$, then we use Theorem~\ref{thm:reg-time} to get $|u^*_{t,\rho}|\lesssim\omega_t(h)$, and argue as in \eqref{eq:U-sigma-est} to derive $|u^*_{t,\rho}|\lesssim\sigma_0(\rho^{1/2})$. On the other hand, when $h\ge\rho$, we apply Theorem~\ref{thm:reg-time} to get
\begin{equation*}
    |u^*_{t,\rho}|\le \frac{|\de_{t,\rho}\uu(t,x)|+|\de_{t,\rho}\uu(t-h,x)|}{h^{1/2}}\lesssim \frac{\rho^{1/2}\omega_t(\rho)}{h^{1/2}}\lesssim\omega_t(\rho).
\end{equation*}
The theorem is proved.
\end{proof}

%%%%%%%%%%%%%%%%%%%%%%%%%%%%%%%%%%%%%%%%%%%%%%%%%%%%%%

\section{Higher order Schauder type estimates}\label{sec:HO}
The objective of this section is to establish higher order Schauder type estimates for systems with partially DMO coefficients, Theorem~\ref{thm:main-reg}. To accomplish this, we first prove a similar result with Hölder coefficients in Section~\ref{subsec:HO-Holder}, and then use this result to achieve Theorem~\ref{thm:main-reg} in Section~\ref{subsec:HO-partiallyDMO}.

\subsection{Systems with Hölder coefficients}\label{subsec:HO-Holder}
As mentioned above, we derive higher-order Schauder type estimates when coefficients belong to Hölder spaces. Specifically, we prove the following:

\begin{theorem}\label{thm:higher-reg-Holder}
    Suppose $k\in \mathbb{N}$ and $0<\be<1$, and let $u$ be a solution of \eqref{eq:pde}. If $A$ satisfies \eqref{eq:assump-coeffi} and $A,\bg\in C^{\frac{k+\be-1}2,k+\be-1}(\overline{Q_1^+})$, then $\uu\in C^{\frac{k+\be}2,k+\be}(\overline{Q^+_{1/2}})$ and $x_n^\al Du\in C^{\frac{k+\be-1}2,k+\be-1}(\overline{Q_{1/2}^+})$.
\end{theorem}

In Theorem~\ref{thm:higher-reg-Holder}, we derive the regularity of $x_n^\al Du$ as well as $\uu$. This is because the regularity of $x_n^\al Du$ is needed in the inductive argument for proving the regularity of $\uu$ for $k\in\mathbb{N}$. We begin the proof by considering the case when $k=1$.

\begin{theorem}
    \label{thm:higher-reg-Holder-1}
    Theorem~\ref{thm:higher-reg-Holder} holds when $k=1$.
\end{theorem}

\begin{proof}
    From the proofs of Theorems~\ref{thm:reg-Diri} and \ref{thm:reg-time}, we infer $D_{x'}\uu=x_n^\al D_{x'}u\in C^{\be/2,\be}$, $x_n^\al D_nu\in C^{\be/2,\be}$, and $\uu\in C^{\frac{1+\be}2}_t$. Moreover, due to \eqref{eq:u_alpha}, $x_n^\al D_nu\in C^{\be/2,\be}$ implies $D_n\uu\in C^{\be/2,\be}$.
\end{proof}

\begin{theorem}
    \label{thm:higher-reg-Holder-2}
    Theorem~\ref{thm:higher-reg-Holder} holds when $k=2$.
\end{theorem}

\begin{proof}
To prove Theorem~\ref{thm:higher-reg-Holder-2}, it suffices to show that $D_{x'}\uu\in C^{\frac{1+\be}2,1+\be}$, $\partial_t\uu\in C^{\be/2,\be}$, $D_n\uu\in C^{\frac{1+\be}2}_t$, $D_{nn}\uu\in C^{\be/2,\be}$, and $x_n^\al D_nu\in C^{\frac{1+\be}2,1+\be}$. 

\medskip\noindent\emph{Step 1.} From \eqref{eq:pde}, we see that $D_{x'}u$ solves
\begin{align}
    \label{eq:pde-tan-deriv}
    \begin{cases}
        D_\g(x_n^\al A^{\g\de}D_\de(D_{x'}u))-x_n^\al\partial_t(D_{x'}u)=\div\bar\bg&\text{in } Q_1^+,\\
        D_{x'}u=0&\text{on }Q_1',
    \end{cases}
\end{align}
where $\bar g_\g:=D_{x'}g_\g-x_n^\al D_{x'} A^{\g\de}D_\de u$, $1\le\g\le n$ (\eqref{eq:pde-tan-deriv} can be justified rigorously by using the finite difference quotient). Since $x_n^\al D u\in C^{\be/2,\be}$ by Theorem~\ref{thm:higher-reg-Holder-1}, we have $\bar\bg\in C^{\be/2,\be}$. We then apply Theorem~\ref{thm:higher-reg-Holder-1} to $D_{x'}u$ to obtain $D_{x'}\uu\in C^{\frac{1+\be}2,1+\be}$ and $x_n^\al DD_{x'}u\in C^{\be/2,\be}$.

\medskip\noindent\emph{Step 2.} For a fixed $h\in (0,1/4)$, we let
\begin{align*}
    u^h(t,x):=\frac{u(t,x)-u(t-h,x)}{h^{1/2}}=\frac{\de_{t,h}u(t,x)}{h^{1/2}}.
\end{align*}
We define $\bg^h$ and $(A^{\g\de})^h$ in a similar way. From \eqref{eq:pde}, we deduce
\begin{align}
    \label{eq:pde-h}
    \begin{cases}
        D_\g(x_n^\al A^{\g\de}D_\de u^h)-x_n^\al \partial_tu^h=\div\hat\bg^h&\text{in }Q_1^+,\\
        u^h=0&\text{on }Q'_1,
    \end{cases}
\end{align}
where $\hat g_\g^h(t,x):=g_\g^h(t,x)-x_n^\al(A^{\g\de})^h(t,x)D_\de u(t-h,x)$, $1\le\g\le n$. Recall $x_n^\al Du\in C^{\be/2,\be}$. Moreover, we have $\|\bg^h\|_{C^{\be/2,\be}}\le C\|\bg\|_{C^{\frac{1+\be}2,1+\be}}$ and $\|(A^{\g\de})^h\|_{C^{\be/2,\be}}\le C\|A^{\g\de}\|_{C^{\frac{1+\be}2,1+\be}}$ for some constant $C>0$, independent of $h$; see Step 2 in the proof of \cite{DonJeo24}*{Theorem~4.2}. Thus, $\|\hat\bg^h\|_{C^{\be/2,\be}}\le C(\|\uu\|_{1,\mu_1}+\|\bg\|_{C^{\frac{1+\be}2,1+\be}})$. Therefore, by Theorem~\ref{thm:higher-reg-Holder-1}, $x_n^\al u^h\in C^{\frac{1+\be}2,1+\be}(\overline{Q^+_{1/2}})$ and $x_n^\al Du^h\in C^{\be/2,\be}(\overline{Q_{1/2}^+})$ with 
\begin{align}
    &\label{eq:u^h-est} \|x_n^\al u^h\|_{C^{\frac{1+\be}2,1+\be}(\overline{Q_{1/2}^+})}\le C(\|\uu\|_{1,\mu_1}+\|\bg\|_{C^{\frac{1+\be}2, 1+\be}}),\\
    &\label{eq:u^h-est-2}  \|x_n^\al Du^h\|_{C^{\be/2,\be}(\overline{Q_{1/2}^+})}\le C(\|\uu\|_{1,\mu_1}+\|\bg\|_{C^{\frac{1+\be}2, 1+\be}}).
\end{align}
Observe that
$$
\uu^h(t,x):=\frac{\uu(t,x)-\uu(t-h,x)}{h^{1/2}}=x_n^\al u^h.
$$
Then, with this identity and \eqref{eq:u^h-est} at hand, we can follow Step 3 in the proof of \cite{DonJeo24}*{Theorem~4.2} to obtain $D_n\uu\in C^{\frac{1+\be}2}_t$ and $\partial_t\uu\in C^{\be/2,\be}$. Moreover, \eqref{eq:u^h-est-2} gives
$$
|x_n^\al D_nu(t+h,x)-2x_n^\al D_nu(t,x)+x_n^\al D_nu(t-h,x)|\le C(\|\uu\|_{1,\mu_1}+\|\bg\|_{C^{\frac{1+\be}2, 1+\be}})h^{\frac{1+\be}2},
$$
thus $x_n^\al D_nu\in C^{\frac{1+\be}2}_t$ by \cite{Ste70}*{Proposition~8 in Chapter 5}.

\medskip\noindent\emph{Step 3.} To complete the proof, it remains to show $D_{nn}\uu\in C^{\be/2,\be}$ and $D(x_n^\al D_nu)\in C^{\be/2,\be}$. To prove them, we first claim that
\begin{align}
    \label{eq:normal-deriv-Holder}
    D_n(x_n^\al D_nu)\in C^{\be/2,\be}.
\end{align}
Indeed, by using $x_n^\al DD_{x'}u\in C^{\be/2,\be}$, $\partial_tu_\al\in C^{\be/2,\be}$, Theorem~\ref{thm:higher-reg-Holder-1}, and \eqref{eq:pde}, we obtain
\begin{align*}
    D_n(x_n^\al A^{n\de}D_\de u)&=-\sum_{\g=1}^{n-1}D_\g(x_n^\al A^{\g\de}D_\de u)+\partial_tu_\al+\div\bg\\
    &=-\sum_{\g=1}^{n-1}(x_n^\al A^{\g\de}D_{\g\de}u+x_n^\al D_\g A^{\g\de}D_\de u)+\partial_t u_\al+\div \bg\in C^{\be/2,\be}.
\end{align*}
This, along with $x_n^\al DD_{x'}u\in C^{\be/2,\be}$, $D_{x'}u_\al\in C^{\frac{1+\be}2,1+\be}$, and $x_n^\al D_nu\in C^{\be/2,\be}$, implies \eqref{eq:normal-deriv-Holder}.

Now, $D(x_n^\al D_nu)\in C^{\be/2,\be}$ follows from \eqref{eq:normal-deriv-Holder} and $x_n^\al DD_{x'}u\in C^{\be/2,\be}$. Moreover, by \eqref{eq:u_alpha}, we have
$$
D_{nn}\uu(X)=D_n(x_n^\al D_nu)+\al\int_0^1D_n(x_n^\al D_nu)(t,x',x_n\tau)\tau^{1-\al}d\tau.
$$
Thus, $D_{nn}\uu\in C^{\be/2,\be}$ by \eqref{eq:normal-deriv-Holder}.
\end{proof}

We now establish Theorem~\ref{thm:higher-reg-Holder}. Thanks to the previous two theorems, it remains to prove the result for $k\ge 3$. We prove it by induction on $k$.

\begin{proof}[Proof of Theorem~\ref{thm:higher-reg-Holder}]
We assume that the theorem holds for $1,2,\ldots, k$ for some $k\ge2$, and prove it for $k+1$. To this end, it is enough to show
\begin{align*}
    &D_{x'}u_\al\in C^{\frac{k+\be}2,k+\be},\quad \partial_tu_\al\in C^{\frac{k+\be-1}2,k+\be-1},\quad D_{nn}u_\al\in C^{\frac{k+\be-1}2,k+\be-1},\\
    &x_n^\al D_nu\in C^{\frac{k+\be}2,k+\be}.
\end{align*}
By using the induction hypothesis for $k$ and arguing as in Step 1 in the proof of Theorem~\ref{thm:higher-reg-Holder-2}, we get $D_{x'}u_\al\in C^{\frac{k+\be}2,k+\be}$. Note that $x_n^\al Du\in C^{\frac{k+\be-1}2,k+\be-1}$ by the induction hypothesis for $k$. Since $u_t=\partial_tu$ solves
\begin{align}
    \label{eq:pde-time-deriv}
    \begin{cases}
        D_\g(x_n^\al A^{\g\de}D_\de u_t)-x_n^\al \partial_tu_t=\div\tilde\bg&\text{in }Q_1^+,\\
        u_t=0&\text{on }Q'_1,
    \end{cases}
\end{align}
where $\tilde g_\g:=\partial_t g_\g-x_n^\al \partial_tA^{\g\de}D_\de u\in C^{\frac{k+\be-2}2,k+\be-2}$, we have $\partial_tu_\al=x_n^\al u_t\in C^{\frac{k+\be-1}2,k+\be-1}$ and $x_n^\al D_n(u_t)\in C^{\frac{k+\be-2}2,k+\be-2}$ by the induction hypothesis for $k-1$. Moreover, by arguing as in Step 3 in the proof of Theorem~\ref{thm:higher-reg-Holder-2}, we obtain $D(x_n^\al D_nu)\in C^{\frac{k+\be-1}2,k+\be-1}$ and $D_{nn}u_\al\in C^{\frac{k+\be-1}2,k+\be-1}$. Finally, combining $D(x_n^\al D_nu)\in C^{\frac{k+\be-1}2,k+\be-1}$ with $\partial_t(x_n^\al D_nu)=x_n^\al D_nu_t\in C^{\frac{k+\be-2}2,k+\be-2}$ implies $x_n^\al D_nu\in C^{\frac{k+\be}2,k+\be}$.    
\end{proof}

%%%%%%%%%%%%%%%%%%%%%%%%%%%%%%%%%%%%%%%%%%%%%%%%%%%%%%%%%%%%%%%%%%%%%%%%%%%%%%%%%%%%%%%%%%%%%%%%%%%%

\subsection{Systems with partially DMO coefficients}\label{subsec:HO-partiallyDMO}

In this section, we complete the proof of our central result on Schauder type estimates, Theorem~\ref{thm:main-reg}. The case $k=1$ directly follows from Theorems \ref{thm:reg-Diri}, \ref{thm:reg-time}, and \ref{thm:diff-quo-conti}. Then, it suffices to prove Theorem~\ref{thm:main-reg} for $k=2$, as the case $k\ge3$ can be established by induction as in the proof of Theorem~\ref{thm:higher-reg-Holder}.

We define $\eta_1, \tilde\eta_1, \hat\eta_1, \bar\eta_1$, and $\omega_{t,1}$ in the same way as in \eqref{eq:Dini-tilde}-\eqref{eq:Dini-omega-time}, but with $\eta_A$ and $\eta_\bg$ replaced by $\eta_{A,1}$ and $\eta_{\bg,1}$. Similarly, we define $\omega_{x,1}(r)$ be as in \eqref{eq:omega_x} with $\hat\eta_A$, $\hat\eta_\bg$, and $\hat\eta$ replaced by $\hat\eta_{A,1}$, $\hat\eta_{\bg,1}$, and $\hat\eta_1$. Note that $\omega_{x,1}(r)\lesssim \sigma_1(r)$ and $\omega_{t,1}(r)\lesssim\sigma_1(r^{1/2})$, where $\sigma_1$ is as in \eqref{eq:mod-conti-sigma}.

\begin{proof}[Proof of Theorem~\ref{thm:main-reg}]
As mentioned above, it is enough to prove the theorem for $k=2$. For this, we need to show $D_{x'}\uu\in \mathring{C}^{1/2,1}$, $\partial_t\uu\in \mathring{C}$, $D_{nn}\uu\in \mathring{C}$, $x_n^\al D_nu\in \mathring{C}^{1/2,1}$, and
\begin{align}
    \label{eq:diff-quo-reg}
    \lim_{h\to0}\frac{\de_{t,h}D_n\uu}{h^{1/2}}(t,x)=0\quad\text{and}\quad\frac{\de_{t,h}D_n\uu}{h^{1/2}}\in C_X(\overline{Q_{1/2}^+})
\end{align}
uniformly in $(t,x)\in \overline{Q_{1/2}^+}$ and $h\in (0,1/4)$, respectively.

We first prove $D_{x'}\uu\in \mathring{C}^{1/2,1}$. To this end, we recall that $D_{x'}u$ satisfies \eqref{eq:pde-tan-deriv} with $\bar g_\g=D_{x'}g_\g-x_n^\al D_{x'}A^{\g\de}D_\de u$. Since $u$ solves \eqref{eq:pde} and $A,\bg\in C^{1/2,1}$, we have $x_n^\al Du\in C^{\be/2,\be}$ for any $0<\be<1$ by Theorem~\ref{thm:higher-reg-Holder}. This gives $\bar g_\g\in\cH$, hence $D_{x'}\uu\in \mathring{C}^{1/2,1}$ and $x_n^\al DD_{x'}u\in \mathring{C}$ by the case when $k=1$.

Next, for $h\in (0,1/4)$, let $u^h$, $(A^{\g\de})^h$, $\bg^h$, and $\uu^h$ be as in Step 2 in the proof of Theorem~\ref{thm:higher-reg-Holder-2}. Then $u^h$ solves \eqref{eq:pde-h} with $\hat g_\g^h(t,x)=g_\g^h(t,x)-x_n^\al (A^{\g\de})^h(t,x)D_\de u(t-h,x)\in\cH$, thus we have by Theorem~\ref{thm:reg-time} that
$$
\frac{|\uu^h(t+h,x)-\uu^h(t,x)|}{h^{1/2}}\lesssim \omega_{t,1}(h),\quad (t+h,x),(t,x)\in Q^+_{3/4}.
$$
This gives 
\begin{align}
    \label{eq:u_alpha-diff-quo-est}
    \frac{|\uu(t+2h,x)-2\uu(t+h,x)+\uu(t,x)|}h\lesssim\omega_{t,1}(h),\quad 0<h<1/4,
\end{align}
which corresponds to \cite{DonJeo24}*{(4.12)}. With \eqref{eq:u_alpha-diff-quo-est} and Theorem~\ref{thm:reg-time} applied to $D_{x'}u$ at hand, we can follow the argument in Step 2 in the proof of \cite{DonJeo24}*{Theorem~1.2} to obtain $\partial_t\uu\in \mathring{C}$. Moreover, since $x_n^\al u^h=0$ on $Q'_{1/2}$ by Remark~\ref{rem:u-alpha-bound}, we can replace $u$ by $u^h$ in \eqref{eq:u_alpha} to get
$$
D_nu^h_\al(X)=(x_n^\al D_nu^h)(X)+\al\left(H^h(X)+\al\int_0^1H^h(X',x_n\tau)\tau^{-\al}d\tau\right),
$$
where $H^h(X)=\int_0^1(x_n^\al D_nu^h)(X',x_n\rho)d\rho$. Since $x_n^\al D_nu^h\in \mathring{C}$, this equality implies $D_nu^h_\al\in \mathring{C}$. In particular, we have the estimate
\begin{align}\label{eq:deriv-h-est}
|D_n\uu^h(t+h)-D_n\uu^h(t)|\lesssim\omega_{x,1}(h^{1/2}).
\end{align}
We then argue as in Step 3 in the proof of \cite{DonJeo24}*{Theorem~1.2} to obtain 
$$
\frac{|\de_{t,h}D_n\uu|}{h^{1/2}}\lesssim \sigma_1(h^{1/2})\quad\text{uniformly in }(t,x).
$$
This implies the uniform convergence in \eqref{eq:diff-quo-reg}. Now, with \eqref{eq:deriv-h-est} and $DD_n\uu\in C_t$, which follows from $D_{x'}\uu\in \mathring{C}^{1/2,1}$ and $D_{nn}\uu\in \mathring{C}$, at hand, we can repeat the proof of Theorem~\ref{thm:diff-quo-conti} with $D_n\uu$ in the place of $\uu$ to conclude $\frac{\de_{t,h}(D_n\uu)}{h^{1/2}}\in C_X(\overline{Q_{1/2}^+})$ uniformly in $h\in (0,1/4)$.

Next, with $x_n^\al DD_{x'}u\in \mathring{C}$, $\partial_t\uu\in \mathring{C}$, $D_{x'}\uu\in\mathring{C}^{1/2,1}$, and $x_n^\al D_nu\in \mathring{C}$ at hand, we can argue as in Step 3 in the proof of Theorem~\ref{thm:higher-reg-Holder-2} to obtain $D_n(x_n^\al D_nu)\in\mathring{C}$ and $D_{nn}\uu\in \mathring{C}$.

It remains to prove $x_n^\al D_nu\in \mathring{C}^{1/2,1}$. Thanks to $x_n^\al D_nD_{x'}u\in\mathring{C}$ and $D_n(x_n^\al D_nu)\in\mathring{C}$, it is sufficient to show
$$
\lim_{h\to0}\frac{\delta_{t,h}(x_n^\al D_nu)}{h^{1/2}}(t,x)=0\quad\text{and}\quad\frac{\delta_{t,h}(x_n^\al D_nu)}{h^{1/2}}\in C_X
$$
uniformly in $(t,x)\in\overline{Q^+_{1/2}}$ and $h\in(0,1/4)$, respectively. To this end, we observe that
$$
(x_n^\al D_nu)^h(t,x):=\frac{(x_n^\al D_nu)(t,x)-(x_n^\al D_nu)(t-h,x)}{h^{1/2}}=(x_n^\al D_nu^h)(t,x)
$$
and recall that $u^h$ is a solution of \eqref{eq:pde-h}. Thus, by Theorem~\ref{thm:reg-Diri}, we have
$$
|(x_n^\al D_nu)^h(t+h,x)-(x_n^\al D_nu)^h(t,x)|\lesssim \omega_{x,1}(h^{1/2}).
$$
We then argue as in Step 3 in the proof of \cite{DonJeo24}*{Theorem~1.2} to get
$$
\frac{|\delta_{t,h}(x_n^\al D_nu)|}{h^{1/2}}\lesssim\sigma_1(h^{1/2})\quad\text{uniformly in }(t,x).
$$
Finally, by using this estimate and the previous result $x_n^\al DD_{x'}u\in\mathring{C}$ and $D_n(x_n^\al D_nu)\in\mathring{C}$, which implies $D(x_n^\al D_nu)\in\mathring{C}$, we can follow the argument in the proof of Theorem~\ref{thm:diff-quo-conti} to conclude that $\frac{\delta_{t,h}(x_n^\al D_nu)}{h^{1/2}}\in C_X$ uniformly in $h$.
\end{proof}

%%%%%%%%%%%%%%%%%%%%%%%%%%%%%%%%%%%%%%%%%%%%%%%%%%%%%%%%%%%%%%%%%%%

\section{Boundary Harnack Principle}\label{sec:BHP}
In this section, we establish the higher-order boundary Harnack principle for degenerate or singular equations, Theorem~\ref{thm:par-BHP-deg}. The proof follows the approach in the corresponding result for uniformly parabolic equations (\cite{DonJeo24}*{Theorem~1.4}), but we provide the full proof since there are technical differences.

\begin{proof}[Proof of Theorem~\ref{thm:par-BHP-deg}]
In the same way as earlier, we write
$$
\uu(X):=x_n^\al u(X),\quad \vv(X):=x_n^\al v(X).
$$
Note that $\uu,\vv\in C^{\frac{k+\be}2,k+\be}(\overline{Q^+_{4/5}})$ by Theorem~\ref{thm:higher-reg-Holder}. Since $\uu=0$ on $Q'_{4/5}$ by Remark~\ref{rem:u-alpha-bound}, this gives
\begin{align}\label{eq:u_alpha/x_n}
u(X)/x_n^{1-\al}=\uu(X)/x_n=\int_0^1\partial_n\uu(X',x_n\tau)d\tau\in C^{\frac{k-1+\be}2,k-1+\be}(\overline{Q^+_{4/5}}).
\end{align}
Similarly, $v/x_n^{1-\al}\in C^{\frac{k-1+\be}2,k-1+\be}$. Moreover, the condition $\partial_n(x_n^\al u)>0$ on $Q'_1$ and the continuity and positivity of $u/x_n^{1-\al}$ in $Q_{4/5}^+$ imply that $u/x_n^{1-\al}\ge c>0$ in $Q^+_{3/4}$. Thus $w:=v/u$ satisfies
$$
\|w\|_{L^\infty(Q^+_{3/4})}\le \|v/x_n^{1-\al}\|_{L^\infty(Q^+_{3/4})}\|(u/x_n^{1-\al})^{-1}\|_{L^\infty(Q_{3/4}^+)}\le C.
$$
We then divide the rest of the proof into three steps.

\medskip\noindent\emph{Step 1.} Suppose $k=1$. We will first derive an a priori Hölder estimate of $\D w$ under the assumption $\D w\in C^{\be/2,\be}(\overline{Q_{3/4}^+})$. The general case follows from a standard approximation argument. By the symmetry of $A$,
\begin{align*}
    \div(x_n^\al u^2A\D w)&=\div(ux_n^\al A\D v-vx_n^\al A\D u)=u\div(x_n^\al A\D v)-v\div(x_n^\al A\D u)\\
    &=u(x_n^\al\partial_tv+g)-v(x_n^\al\partial_tu+f)=x_n^\al u^2\partial_tw+ug-vf.
\end{align*}
This gives
\begin{align*}
    \div(x_n^{2-\al}A\D w)&=\div\left(\left(u/x_n^{1-\al}\right)^{-2}x_n^\al u^2A\D w\right)\\
    &=\left(u/x_n^{1-\al}\right)^{-2}\div(x_n^\al u^2A\D w)+\lmean{\D\left(\left(u/x_n^{1-\al}\right)^{-2}\right),x_n^\al u^2A\D w}\\
    &=\left(u/x_n^{1-\al}\right)^{-2}(x_n^\al u^2\partial_tw+ug-vf)\\
    &\qquad+2x_n^{1-\al}\lmean{A\left((1-\al)\vec{e}_n-\left(u/x_n^{1-\al}\right)^{-1}x_n^\al\D u\right),\D w}\\
    &=x_n^{2-\al}\partial_tw+x_n^{1-\al}\tilde g,
\end{align*}
where
\begin{align}\label{eq:rhs-w}\begin{split}
    \tilde g(X',x_n)&:=\left(u/x_n^{1-\al}\right)^{-1}g-\left(u/x_n^{1-\al}\right)^{-2}\left(v/x_n^{1-\al}\right)f\\
    &\qquad+\lmean{2A\left((1-\al)\vec{e}_n-\left(u/x_n^{1-\al}\right)^{-1}x_n^\al\D u\right),\D w}.
\end{split}\end{align}
We write $\tilde g=\tilde g_1+\mean{\bb,\D w}$, where $\tilde g_1=(u/x_n^{1-\al})^{-1}g-(u/x_n^{1-\al})^{-2}(v/x_n^{1-\al})f$ and $\bb=2A((1-\al)\vec{e}_n-(u/x_n^{1-\al})^{-1}x_n^\al\D u)$. We have
\begin{align}\label{eq:pde-w}
\div(x_n^{2-\al}A\D w)-x_n^{2-\al}\partial_tw=\div(x_n^{2-\al}\tilde\bg),
\end{align}
where
$$
\tilde\bg(X',x_n):=\frac{\vec{e}_n}{x_n^{2-\al}}\int_0^{x_n}\rho^{1-\al}\tilde g(X',\rho)d\rho=\vec{e}_n\int_0^1\tau^{1-\al}\tilde g(X',x_n\tau)d\tau.
$$
Since $2-\al>1$, $w$ satisfies the conormal boundary condition
$$
\lim_{x_n\to0+}x_n^{2-\al}\mean{A\D w-\tilde\bg,\vec{e}_n}=0.
$$
Note that $[\tilde\bg]_{C^{\be/2,\be}(\overline{Q_r^+(X_0)})}\le C[\tilde g]_{C^{\be/2,\be}(\overline{Q_r^+(X_0)})}\le C$ whenever $X_0\in Q'_{2/3}$ and $Q_r^+(X_0)\subset Q^+_{3/4}$. Then, by following the arguments in \cite{Don12}*{Section~5}, we obtain that for every $X_0\in Q^+_{2/3}$ and $0<\rho<r<1/96$,
\begin{align*}
    \tilde\psi(X_0,\rho)&\le C(\rho/r)^\be\|\D w\|_{L^\infty(Q^+_{5r}(X_0))}\\
    &\qquad+C\left(\|\D w\|_{L^\infty(Q_{8r}^+(X_0))}[A]_{C^{\be/2,\be}(\overline{Q_{8r}^+(X_0)})}+[\tilde\bg]_{C^{\be/2,\be}(\overline{Q^+_{8r}(X_0)})}\right)\rho^\be,
\end{align*}
where $\tilde\psi(X_0,r)=\left(\dashint_{Q_r^+(X_0)}|\D w-\mean{\D w}_{Q_r^+(X_0)}|^2d\tilde\mu\right)^{1/2}$, $d\tilde\mu:=x_n^{2-\al}dX$. This implies that for any $0<r_0<R_0<1/6$, $X\in Q^+_{r_0}$ and $0<\rho<\frac{R_0-r_0}{16}$,
$$
\tilde\psi(X,\rho)\le CM\rho^\be,\quad M:=\frac{\|\D w\|_{L^\infty(Q_{R_0}^+)}}{(R_0-r_0)^\be}+[\tilde g]_{C^{\be/2,\be}(\overline{Q^+_{R_0}})}
$$
for some constant $C=C(n,\al,\la,\be,[A]_{C^{\be/2,\be}(\overline{Q^+_{3/4}})})>0$. By the standard Campanato space embedding (see e.g. the proof of Theorem~3.1 in \cite{HanLin97}), we have
\begin{align}
    \label{eq:grad-w-holder}\begin{split}
    [\D w]_{C^{\be/2,\be}(\overline{Q_{r_0}^+})}&\le C\left(M+\frac{\|\D w\|_{L^\infty(Q^+_{R_0})}}{(R_0-r_0)^\be}\right)\\
    &\le C\left(\frac{\|\D w\|_{L^\infty(Q^+_{R_0})}}{(R_0-r_0)^\be}+[\tilde g]_{C^{\be/2,\be}(\overline{Q^+_{R_0}})}\right).
\end{split}\end{align}

We claim that $\bb\in C^{\be/2,\be}$ and $|\bb|\le Cx_n^\be$ in $\overline{Q^+_{3/4}}$. Indeed, since $u/x_n^{1-\al}\ge c$, $\uu\in C^{\frac{1+\be}2,1+\be}(\overline{Q^+_{3/4}})$ and $\uu=0$ on $Q'_{1/2}$, we have
$$
\left|\left(u/x_n^{1-\al}\right)^{-1}x_n^\al\D_{x'}u\right|=\left|\left(u/x_n^{1-\al}\right)^{-1}\D_{x'}\uu\right|\le Cx_n^\be
$$
and
\begin{align*}
    \left|(1-\al)\left(u/x_n^{1-\al}\right)-x_n^\al D_nu\right|&=\left|\uu/x_n-\left(\al\left(u/x_n^{1-\al}\right)+x_n^\al D_nu\right)\right|=\left|\uu/x_n-D_n\uu\right|\\
    &\le \int_0^1|D_n\uu(X',x_n\tau)-D_n\uu(X',x_n)|d\tau\le Cx_n^\be.
\end{align*}
These two estimates imply $|\bb|\le Cx_n^\be$. In addition, it can be inferred from the above computations and \eqref{eq:u_alpha/x_n} that
$$
\bb=2A\left(-\left(u/x_n^{1-\al}\right)^{-1}\D_{x'}\uu,1-\frac{D_n\uu}{\uu/x_n} \right)\in C^{\be/2,\be}.
$$
The claim is proved.

By using the claim above, we get for every $0<r_0<R_0<1/6$,
$$
[\tilde g]_{C^{\be/2,\be}(\overline{Q^+_{R_0}})}\le [\tilde g_1]_{C^{\be/2,\be}(\overline{Q_{3/4}^+})}+C\|\D w\|_{L^\infty(Q_{R_0}^+)}+CR_0^\be[\D w]_{C^{\be/2,\be}(\overline{Q_{R_0}^+})}.
$$
Combining this with \eqref{eq:grad-w-holder} and using the interpolation inequality yield 
\begin{align*}
    [\D w]_{C^{\be/2,\be}(\overline{Q_{r_0}^+})}&\le C\left(\frac{\|\D w\|_{L^\infty(Q_{R_0}^+)}}{(R_0-r_0)^\be}+[\tilde g_1]_{C^{\be/2,\be}(\overline{Q_{3/4}^+})}\right)+CR_0^\be[\D w]_{C^{\be/2,\be}(\overline{Q_{R_0}^+})}\\
    &\le C\left(\frac{\|w\|_{L^\infty(Q_{3/4}^+)}}{(R_0-r_0)^{2+\be}}+[\tilde g_1]_{C^{\be/2,\be}(\overline{Q_{3/4}^+})}\right)+CR_0^\be[\D w]_{C^{\be/2,\be}(\overline{Q_{R_0}^+})}.
\end{align*}
We fix $S_0\in (0,1/6)$ so that $CS_0^\be<1/2$ and apply \cite{Gia83}*{Lemma~3.1 of Ch.V} to deduce
$$
[\D w]_{C^{\be/2,\be}(\overline{Q^+_{S_0/2}})}\le C\left(\|w\|_{L^\infty(Q_{3/4}^+)}+[\tilde g_1]_{C^{\be/2,\be}(\overline{Q_{3/4}^+})}\right).
$$
By varying the centers $X_0'\in Q'_{2/3}$, we further have $\D w\in C^{\be/2,\be}(\overline{Q_{2/3}^+\cap\{x_n<S_0/2\}})$. Moreover, since $w$ satisfies $\div(x_n^{2-\al}A\D w)-x_n^{2-\al}\partial_tw=x_n^{1-\al}\tilde g_1+\mean{x_n^{1-\al}\bb,\D w}$ in $\cQ:=\overline{Q_{2/3}^+}\cap\{x_n>S_0/4\}$ and $0<c<x_n^{2-\al}<C$ in $\cQ$, we have $\D w\in C^{\be/2,\be}(\cQ)$ by \cite{Don12}*{Theorem~2}. Therefore, $\D w\in C^{\be/2,\be}(\overline{Q^+_{2/3}})$. Then $w$ can be seen as a solution of \eqref{eq:pde-w} in $Q^+_{2/3}$ with $\tilde\bg\in C^{\be/2,\be}(\overline{Q^+_{2/3}})$. Thus $w\in C^{\frac{1+\be}2,1+\be}(\overline{Q^+_{1/2}})$ by \cite{DonJeo24}*{Thoerem~4.1}.

\medskip\noindent\emph{Step 2.} 
In this step, we prove the theorem when $k=2$. To this end, we need to show $\D_{x'}w\in C^{\frac{1+\be}2,1+\be}$, $D_nw\in C_t^{\frac{1+\be}2}$, $\partial_tw\in C^{\be/2,\be}$, and $D_{nn}w\in C^{\be/2,\be}$. From \eqref{eq:pde-w}, we infer
$$
\div(x_n^{2-\al}A\D(\D_{x'}w))-x_n^{2-\al}\partial_t(\D_{x'}w)=\div(x_n^{2-\al}(\D_{x'}\tilde\bg-\D_{x'}A\D w)).
$$
Note that
$$
\D_{x'}\tilde\bg=\vec{e}_n\int_0^1\tau^{1-\al}(\D_{x'}g_1+\mean{\D_{x'}\bb,\D w}+\mean{\bb,\D(\D_{x'}w)})(X',x_n\tau)d\tau.
$$
As $\D_{x'}A\D w$, $\D_{x'}g_1+\mean{\D_{x'}\bb,\D w}\in C^{\be/2,\be}$, we have $\D_{x'}w\in C^{\frac{1+\be}2,1+\be}$ by Step 1.

Next, we observe that $w^h(t,x):=\frac{w(t,x)-w(t-h,x)}{h^{1/2}}$ satisfies
$$
\div(x_n^{2-\al}A\D w^h)-x_n^{2-\al}\partial_tw^h=\div(x_n^{2-\al}\hat\bg^h)\quad\text{in }Q_1^+,
$$
where $\hat\bg^h:=\tilde\bg^h-A^h(X)\D w(t-h,x)$. Note that
\begin{align*}
    \tilde\bg^h(X)=\vec{e}_n\int_0^1\tau^{1-\al}(\mean{\bb,\D w^h}+\tilde g_1^h)(X',\tau x_n)+\mean{\bb^h(X',\tau x_n), \D w(t-h,x',\tau x_n)})d\tau.
\end{align*}
Since $A^h, \D w,\tilde g_1^h,\bb^h\in C^{\be/2,\be}$, we obtain by Step 1 that $w^h\in C^{\frac{1+\be}2,1+\be}$. This implies $D_nw\in C^{\frac{1+\be}2}_t$ and $\partial_tw\in C^{\be/2,\be}$; see Step 3 in the proof of \cite{DonJeo24}*{Theorem~4.2}.

Finally, to show $D_{nn}w\in C^{\be/2,\be}$, we recall $\div(x_n^{2-\al}A\D w)=x_n^{2-\al}\partial_tw+x_n^{1-\al}\tilde g_1+x_n^{1-\al}\mean{\bb,\D w}$. We set $W:=\mean{A\D w,\vec{e}_n}-\frac{\tilde g_1}{2-\al}\in C^{\be/2,\be}(\overline{Q^+_{1/2}})$. It is easily seen that
\begin{align*}
    x_n^{\al-2}D_n(x_n^{2-\al}W)=\partial_tw+\mean{\bb/x_n,\D w}-\sum_{i=1}^{n-1}D_i(\mean{A\D w,\vec{e}_i})-\frac{D_n\tilde g_1}{2-\al}=:\bar f.
\end{align*}
Since $W\in C^{\be/2,\be}(\overline{Q_{1/2}^+})$ implies $x_n^{2-\al}W=0$ on $Q_{1/2}'$, this equality gives
$$
W=\frac1{x_n^{2-\al}}\int_0^{x_n}\rho^{2-\al}\bar f(X',\rho)d\rho,
$$
and thus
\begin{align}\label{eq:W}
D_nW=\bar f+(\al-2)\int_0^1\rho^{2-\al}\bar f(X',x_n\rho)d\rho.
\end{align}
Since $\frac{\D_{x'}u(X)}{x_n^{1-\al}}=\frac{\D_{x'}\uu(X)}{x_n}=\int_0^1D_n\D_{x'}\uu(X',x_n\rho)d\rho\in C^{\be/2,\be}$ and $\frac1{x_n}\left(\frac{(1-\al)u(X)}{x_n^{1-\al}}-x_n^\al D_nu(X)\right)=\frac1{x_n}\left(\frac{\uu(X)}{x_n}- D_n\uu(X)\right)=-\int_0^1\rho D_{nn}\uu(X',x_n\rho)d\rho\in C^{\be/2,\be}$, we have $\bb/x_n\in C^{\be/2,\be}$ by \eqref{eq:u_alpha/x_n}. Moreover, from \eqref{eq:u_alpha/x_n}, $u/x_n^{1-\al}\ge c>0$, and $f,g\in C^{\frac{1+\be}2,1+\be}$, we see that $\tilde g_1\in C^{\frac{1+\be}2,1+\be}$. Thus, $\bar f\in C^{\be/2,\be}$, hence $D_nW\in C^{\be/2,\be}$ by \eqref{eq:W}. Therefore, we conclude $D_{nn}w\in C^{\be/2,\be}$.

\medskip\noindent\emph{Step 3.} The case when $k\ge 3$ follows from an induction argument as in the proof of \cite{DonJeo24}*{Theorem~4.1}. 
\end{proof}

%%%%%%%%%%%%%%%%%%%%%%%%%%%%%%%%%%%%%%%%%%%%%%%%%%%%%%%%%%%%%%%%%%%%%%%%%%%%%%%

\end{document}